\documentclass[11pt,reqno]{amsart}
\usepackage{amssymb,latexsym}
\usepackage[sorted,compressed-cites,sorted-cites,initials]{amsrefs}

\usepackage[cmtip,all]{xy}
\usepackage{geometry}
\usepackage{enumerate}
\usepackage{hyperref}
\numberwithin{equation}{section}
\usepackage{pifont,mathrsfs}
\newcommand{\lr}[1]{\langle#1\rangle}
\makeatletter
\def\lort{{\mathrel{\hbox{\hglue .2ex
  \vrule \@height .07ex \@width 0.64ex
  \vrule \@height 1.28ex \@width .07ex
  \hglue .2ex}}}}

\def\rort{{\mathrel{\hbox{\hglue .4ex
  \vrule \@height 1.28ex \@width .07ex
  \vrule \@height .07ex \@width 0.64ex
  \hglue .2ex
}}}}
\def\lrort{{\mathrel{\hbox{\hglue .2ex
  \vrule \@height .07ex \@width 0.64ex
  \vrule \@height 1.28ex \@width .07ex\vrule \@height .07ex \@width 0.64ex
  \hglue .2ex}}}}

\makeatother

\newtheorem{theorem}{Theorem}[section]
\newtheorem{proposition}[theorem]{Proposition}

\newtheorem{corollary}[theorem]{Corollary}

\newtheorem{lemma}[theorem]{Lemma}

\theoremstyle{definition}
\DeclareMathOperator{\Ext}{Ext}
\DeclareMathOperator{\Hom}{Hom}
\DeclareMathOperator{\Ob}{Ob}
\DeclareMathOperator{\Spec}{Spec}

\DeclareMathOperator{\Iso}{Iso}

\DeclareMathOperator{\Aut}{Aut}
\DeclareMathOperator{\Ind}{Ind}
\DeclareMathOperator{\GL}{GL}

\DeclareMathOperator{\add}{Add}

\newcommand{\gr}[1]{|#1|}
\newcommand{\cat}[1]{{\mathscr #1}}

\newtheorem{definition}[theorem]{Definition}
\DeclareMathOperator{\ad}{ad}

\newcommand{\tensor}{\otimes}

\def\FF{\mathbb{F}}
\def\ZZ{\mathbb{Z}}
\def\QQ{\mathbb{Q}}

\DeclareMathOperator{\Exp}{Exp}
\DeclareMathOperator{\End}{End}
\def\kk{\Bbbk}

\newcommand{\qbinom}[3][q]{\genfrac[]{0pt}0{#2}{#3}_{#1}}
\DeclareMathOperator{\Mod}{mod}

\advance\oddsidemargin by-0.5in
\advance\evensidemargin by-0.5in
\advance\textwidth by 1in

\begin{document}
\newgeometry{margin=2.5cm}
\author[A. Berenstein]{Arkady Berenstein}
\address{Department of Mathematics, University of Oregon,
Eugene, OR 97403, USA} \email{arkadiy@math.uoregon.edu}

\author[J. Greenstein]{Jacob Greenstein}
\address{Department of Mathematics, University of California, Riverside, CA 92521.} 
 \email{jacob.greenstein@ucr.edu}

\date{\today}

\thanks{The authors were supported in part
by NSF grants DMS-0800247, DMS-1101507 (A.B.) and DMS-0654421 (J.~G.)
}

\title{Quantum Chevalley groups}

\begin{abstract} 
The goal of this paper is to construct quantum analogues of Chevalley groups inside 
completions of quantum groups or, more precisely, inside completions of Hall algebras of finitary categories.
In particular, we obtain pentagonal and other identities in the quantum Chevalley groups which generalize their classical counterparts and 
explain Faddeev-Volkov quantum dilogarithmic identities and their recent generalizations due to Keller
\end{abstract}

\maketitle

\section{Introduction}
\label{sect:intro}

It is well-known that a quantum group is not a group, in particular, it has very few elements that can be regarded as ``group-like''. To construct group-like 
elements, it suffices to introduce a collection of ``exponentials'' in the completed quantum group which are analogues of formal exponentials $\exp(x)$, where $x$ runs over a 
Lie algebra $\mathfrak g$ and the exponential $\exp(x)$ is viewed as an element of the completed enveloping algebra $\widehat U(\mathfrak g)$. 
Since each quantum group can be realized as a (twisted) Hall-Ringel algebra of a (finitary) abelian category, this motivates the following definition. 
Consider the completion~$\widehat H_{\cat A}$ of the Hall algebra $H_{\cat A}$ of~$\cat A$ with respect to its standard grading 
by the Grothendieck group of~$\cat A$.

\begin{definition} Let ${\cat A}$ be a finitary abelian category. Define the {\it exponential}  $\Exp_{\cat A}$ of ${\cat A}$  in the completed Hall algebra $\widehat H_{\cat A}$ by:
$$\Exp_{\cat A}=\sum_{[A]\in\Iso\cat A}[A] $$ 
where $\Iso\cat A$ is the set of isomorphism classes $[A]$ of objects in $\cat A$.
\end{definition}

We can justify this definition, in particular, by the following simple observations.

\begin{lemma} For any finitary category $\cat A$ one has
$$\Delta(\Exp_{\cat A})=\Exp_{\cat A}\otimes \Exp_{\cat A},$$ 
where $\Delta:H_{\cat A}\to H_{\cat A}\widehat \otimes H_{\cat A}$ is a coproduct (though, not always co-associative) defined in Appendix~\ref{sect:Appendix}.
\end{lemma}

\begin{lemma} If ${\cat A}$ is the category of finite dimensional vector spaces over a finite field $\kk$, then 
$$\Exp_{\cat A}=\exp_q(x)=\sum_{n\ge 0} \frac {x^n}{[n]_q!}$$
where $x$ is the isomorphism class of $\kk$ in ${\cat A}$, $q=|\kk|$, and $[n]_q!=[1]_q[2]_q\cdots [n]_q$, where $[k]_q=1+q+\cdots +q^{k-1}$ is the $q$-number.
\end{lemma}

In general the exponentials behave nicely due to the following extension of Reineke's inversion formula \cite{reineke}*{Lemma 3.4} to more general abelian categories.
\begin{theorem} 
For any finitary abelian category with the finite length property
$$
{\Exp_{\cat A}}^{-1}=\sum_{[M]\,:\,\text{$M$ \em semisimple}} c_{[M]} \cdot [M],
$$
where
$$
c_{[M]}=\prod_{[S]\,:\,\text{$S$ \em simple}} (-1)^{[M:S]} |\End_{\cat A} S|^{\binom{[M:S]}{2}}
$$
and~$[M:S]$ denotes the Jordan-H\"older multiplicity of~$S$ in~$M$.
\end{theorem}

These observations suggest the following definition.
\begin{definition} Given a finitary abelian category ${\cat A}$, its  {\it quantum Chevalley group} $G_{\cat A}$ is the subgroup of $(\widehat H_{\cat A})^\times$ generated by all exponentials $\Exp_{\cat C}\in \widehat H_{\cat C}\subset \widehat H_{\cat A}$, 
where ${\cat C}$ runs over  all full subcategories of~$\cat A$ closed under extensions.
\end{definition}

Now we investigate the adjoint action of some elements of quantum Chevalley groups on their Hall algebras. 

For each object $E$ of ${\cat A}$ denote by $\add_{\cat A}(E)$ the minimal additive full subcategory of ${\cat A}$ containing all direct 
summands of~$E$. The following result  is obvious.
\begin{lemma} 
\label{le:brick}
If $E$ is an object of ${\cat A}$ without self-extensions such that $\End_{\cat A}E$ is a field, then $\add_{\cat A}(E)$ is semisimple and the element 
$$\Exp_{\add_{\cat A}(E)}=\exp_{q_E}([E])$$
of $\widehat H_{\cat A}$, where $q_E=|\End_{\cat A}E|$, belongs to the quantum Chevalley group. 

\end{lemma} 

Given a full additive subcategory $\cat B$ of ${\cat A}$ we denote by $\cat B^\rort$ (respectively $\cat B^\lort$) the largest full subcategory $\cat C$ of $\cat A$ 
whose objects~$C$ satisfy $\Ext^1_{\cat A}(B,C)=0$ (respectively $\Ext^1_{\cat A}(C,B)=0$) for all $B\in\Ob\cat B$.
Clearly, ${\cat C}\subset {\cat B}^\rort$ if
and only if ${\cat B}\subset {\cat C}^\lort$. The following Lemma is immediate.

\begin{lemma} For any full additive subcategory~$\cat B$ of~$\cat A$, $\cat B^\lort$ and $\cat B^\rort$ are closed under extensions.
\end{lemma}

We will show below that  the conjugation with exponentials $\Exp_{\add_{\cat A}(E)}=\exp_{q_E}([E])$ for each $E$ as in Lemma \ref{le:brick} almost preserves the entire Hall algebra  of the categories $\add_{\cat A}(E)^\rort$ and $\add_{\cat A}(E)^\lort$.

To make this statement
precise, we need more notation. Let~$R$ be a ring and let $S=\{a_\iota\}_{\iota\in I}$ be a set of pairwise commuting 
regular elements (i.e., not zero divisors) in $R$. The {\em localization} $R[S^{-1}]$ of~$R$ with respect to~$S$ is then the quotient of the free 
product $R*{\mathbb Z}[x_\iota:\iota\in I]$ by the ideal generated by $a_\iota*x_\iota-1$ and $x_\iota*a_\iota-1$, $\iota\in I$.

Clearly, for each $E$ as in Lemma \ref{le:brick} the subset $\mathcal S_E\subset H_{\add_{\cat A}(E)^\lort}
\cap H_{\add_{\cat A}(E)^\rort}$ defined by 
$$\mathcal S_E:=\{1+(q_E-1)q_E^j[E]\,|\,j\in\mathbb Z\}$$ 
consists of pairwise commuting regular elements, therefore the localizations $H_{\add_{\cat A}(E)^\lort}[\mathcal S_E^{-1}]$ 
and $H_{\add_{\cat A}(E)^\rort}[\mathcal S_E^{-1}]$ are well-defined.

The following results furthers the analogy between the classical Chevalley group viewed as automorphisms of a semisimple Lie algebra  over $\ZZ$ and its quantum counterpart whose subgroup  
acts by automorphisms of the (slightly extended) Hall subalgebra.

\begin{theorem} 
\label{th:local conjugation}
Let $E$ be as in Lemma \ref{le:brick}. Then the assignment
$$x\mapsto \exp_{q_E}([E]) \cdot x \cdot \exp_{q_E}([E])^{-1}$$  
defines an automorphism of the localized algebras $H_{\add_{\cat A}(E)^\lort}[\mathcal S_E^{-1}]$ and $H_{\add_{\cat A}(E)^\rort}[\mathcal S_E^{-1}]$.
\end{theorem}

We refine and prove the theorem in Section~\ref{sec:comm} (Theorem~\ref{prop:auto}). 
In fact, we show that $\mathcal S_E$ is an Ore set for both $H_{\add_{\cat A}(E)^\lort}$ and $H_{\add_{\cat A}(E)^\rort}$. In a special case
when $E$ is a projective or an injective indecomposable, or, more generally, a preprojective or preinjective indecomposable 
in a hereditary category, 
it turns out that we obtain an automorphism of the localized Hall algebra
$H_{\cat A}[\mathcal S_E^{-1}]$.

Next we discuss factorization of exponentials.
\begin{definition}
We say that an ordered pair of full subcategories $(\cat A',\cat A'')$ of~$\cat A$
is a {\em factorizing pair} if $\cat A',\cat A''$ are closed under extensions and every object~$M$ in~$\cat A$ has a unique subobject~$M''\in\Ob\cat A''$
such that $M/M''$ is an object in~$\cat A'$.

More generally, we define a factorizing sequence $({\cat A}_1,\ldots,{\cat A}_m)$, $m\ge 2$ for ${\cat A}$ recursively by:

$\bullet$ If $m\ge 3$, then $({\cat A}_1,\ldots,{\cat A}_{m-1})$ is a  factorizing sequence 
for the category ${\cat A}_{\le m-1}$, which is the minimal additive subcategory of  
${\cat A}$ containing ${\cat A}_1,\ldots,{\cat A}_{m-1}$ and closed under extensions;

$\bullet$ $({\cat A}_{\le m-1},{\cat A}_m)$ is a factorizing pair for ${\cat A}$.

\end{definition}
The following result is immediate.

\begin{proposition}
\label{pr:factorizable sequence}
If~$({\cat A}_1,\ldots,{\cat A}_m)$ is a factorizing sequence for a finitary abelian category~$\cat A$
then 
\begin{enumerate}[{\rm(a)}]
\item\label{pr:fs.a} $\Exp_{\cat A}=\Exp_{{\cat A}_1}\cdots \Exp_{{\cat A}_m}$.
 
\item\label{pr:fs.b} If, additionally, ${\cat A}_i\subset {\cat A}_j^\lort$ for all $i<j$, then 
$$H_{\cat A}\cong H_{{\cat A}_1}\tensor \cdots \tensor H_{{\cat A}_m}$$
as a vector space.
\end{enumerate}
\end{proposition}
The simplest example of a factorizing pair is provided by the following.
\begin{proposition}
Let $S$ be a simple  object in a finitary abelian category~$\cat A$ with the finite length property. Let~$\cat A'=\add_{\cat A}S$ and let $\cat A''$ be the 
maximal full subcategory of~$\cat A$ whose objects~$M$ satisfy
$[M:S]=0$. 

\begin{enumerate}[{\rm(a)}]
 \item 
 If $S$ is injective, then  $(\cat A',\cat A'')$ is a factorizing pair for~$\cat A$.

\item If $S$ is projective, then  $(\cat A'',\cat A')$  is a factorizing pair for~$\cat A$.
\end{enumerate}
\end{proposition}

Let~$(\cat A',\cat A'')$ be a factorizing pair for~$\cat A$. We say that $(\cat A',\cat A'')$ is a pentagonal pair
if there exists a full additive subcategory~$\cat A_0$ closed under extensions such that 
$(\cat A'',\cat A_0,\cat A')$ is a factorizing triple. If $(\cat A',\cat A'')$ is a pentagonal pair then, clearly
$$
\Exp_{\cat A'}\Exp_{\cat A''}=\Exp_{\cat A}=\Exp_{\cat A''}\Exp_{\cat A_0}\Exp_{\cat A'},
$$
which is the generalized pentagonal relation in the quantum Chevalley group.
\begin{proposition}\label{prop:pent.pair}
Let~$\cat A$ be a finitary abelian category with the finite length property. Let~$\mathcal S\subset \Iso\cat A$
be the set of isomorphism classes of simples in~$\cat A$ and suppose that $\mathcal S=\mathcal S_-\bigsqcup
\mathcal S_+$ where $\mathcal S_-$ (respectively $\mathcal S_+$) is the set of isomorphism classes of projective
(respectively, injective) simple objects. Then $(\cat A_+,\cat A_-)$, where $\cat A_\pm=\add_{\cat A}\mathcal S_\pm$,
is a pentagonal pair with $\cat A_0$ being the full subcategory of~$\cat A$ whose objects have no simple direct summands.
\end{proposition}

In particular, if $\cat A$ has only two isomorphism classes of simple objects $[S_0],[S_1]$ such that 
$\Ext^1_{\cat A}(S_0,S_1)=0$, we obtain
$$
[\exp_{q_0}([S_0]),\exp_{q_1}([S_1])]=\Exp_{\cat A_0}
$$
where $[x,y]=x^{-1}y xy^{-1}$ and~$q_i=|\End_{\cat A}S_i|$. Let~$a_i=\dim_{\End_{\cat A}S_i}\Ext^1_{\cat A}(S_1,S_0)$.
If~$a_0a_1<4$,
$\cat A$ has finitely many isomorphism classes of indecomposables and we obtain
$$
[\exp_{q_0}(X_{\alpha_0}),\exp_{q_1}(X_{\alpha_1})]=\begin{cases}\exp_{q_1}(X_{\alpha_{10}}),&a_0a_1=1,\\
\exp_{q_1}(X_{\alpha_{10}})\exp_{q_0}(X_{\alpha_{01}}),
&a_0a_1=2,\\
\exp_{q_1}(X_{\alpha_{10}})
\exp_{q_0}(X_{\alpha_{01}+\alpha_0})\exp_{q_1}(X_{\alpha_1+\alpha_{10}})\exp_{q_0}(X_{\alpha_{01}}),& a_0a_1=3,
\end{cases}
$$
where~$\alpha_{ij}=\alpha_i+a_j\alpha_j$
(the notation~$X_\alpha$ is explained in Theorem~\ref{thm:descr-prepr-Hall} below).
Other examples of pentagonal pairs are provided in Section~\ref{sec:examples}.

Given $[M]\in\Iso\cat A$, denote by $\gr{M}$ its image in the Grothendieck group~$K_0(\cat A)$ of~$\cat A$. If~$\cat A$ is hereditary, we can consider
the twisted group algebra $\mathcal P_{\cat A}$ of~$K_0(\cat A)$ with the multiplication defined by
$$
t^\alpha t^\beta=\lr{\alpha,\beta}^{-1} t^{\alpha+\beta}
$$
where $\lr{\gr{M},\gr{N}}=|\Hom_{\cat A}(M,N)|/|\Ext^1_{\cat A}(M,N)|$ is the Ringel form. Then we have 
an algebra homomorphism (\cites{Keller,Fei,reineke})
$$
\int:H_{\cat A}\to \mathcal P_{\cat A},\qquad  [M]\mapsto \frac{t^{\gr{M}}}{|\Aut_{\cat A}(M)|}. 
$$
Applying it to pentagonal identities yields the so-called quantum dilogarithm identities (cf.~\cite{Keller}). 
In this case, $\mathcal P_{\cat A}$ is generated by $x_1,x_0$ subject to the relation
$$
x_1 x_0=q^{a_0a_1} x_0 x_1
$$
and we have the following identities in the natural completion~$\widehat{\mathcal P}_{\cat A}$ of~$\mathcal P_{\cat A}$
$$
[\mathbb E_{q_0}(x_0),\mathbb E_{q_1}(x_1)]=\begin{cases}\mathbb E_{q_1}(x_0x_1),&a_0a_1=1\\
\mathbb E_{q_1}(p_0x_0^{a_0}x_1^{})\mathbb E_{q_0}(p_1x_0^{} x_1^{a_1}),
&a_0a_1=2\\
\mathbb E_{q_1}(p_0 x_0^{a_0}x_1^{})\mathbb E_{q_0}(q_0^{}p_1x_0^2 x_1^{a_1})
\mathbb E_{q_1}(p_0q_1^{} x_0^{a_0}x_1^2)\mathbb E_{q_0}(p_1 x_0^{} x_1^{a_1}),
&a_0a_1=3,
\end{cases}
$$
where $\mathbb E_q(x)=\exp_q(x/(q-1))$ is a quantum dilogarithm and $p_i=q_i^{\binom{a_i}2}$.

Let $\cat A$ be the category of finite length modules over a finite hereditary ring and denote by $\cat A_-$ (respectively $\cat A_+$) the preprojective (respectively, the preinjective) component of $\cat A$.
Note that if  $\cat A$ is {\it hereditary} then  $\cat A_-=\cat A_+=\cat A$  if and only if $\cat A$ is of finite type. Otherwise, 
$\cat A_-\subset \cat A_+^\lort$ and
the triple $({\cat A}_-,{\cat A}_0,{\cat A}_+)$, where ${\cat A}_0=\cat A_-{}^\rort\cap \cat A_+{}^\lort$,   is a factorizing sequence for ${\cat A}$
and satisfies hypotheses of Proposition \ref{pr:factorizable sequence}\eqref{pr:fs.b}. 
This implies the following result.
\begin{corollary}  
\label{cor:hereditary factorization}
Suppose that $\cat A$ is hereditary of infinite type. Then
\begin{enumerate}[{\rm(a)}]
\item\label{cor:hered.a} The Hall algebra $H_{\cat A}$ admits a triangular decomposition $H_{\cat A}\cong H_{{\cat A}_-}\otimes H_{{\cat A}_0}\otimes H_{{\cat A}_+}$.

\item\label{cor:hered.b} The exponential $\Exp_{\cat A}$ admits a factorization $\Exp_{\cat A}=\Exp_{{\cat A}_-}\Exp_{{\cat A}_0}\Exp_{{\cat A}_+}$.
\end{enumerate}
\end{corollary}

Finally, we describe $H_{{\cat A}_\pm}$ and $\Exp_{{\cat A}_\pm}$ (and, to some extent, $H_{{\cat A}_0}$ and $\Exp_{{\cat A}_0}$) in the case when $\cat A$ is the representation category of a finite hereditary algebra over~$\kk$.

Number the isomorphism classes of simples $[S_i]$, $1\le i\le r$ in such a way that $\Ext^1_{\cat A}(S_j,S_i)=0$ if~$i<j$.
The elements~$\alpha_i=\gr{S_i}$ form a basis of~$K_0(\cat A)$. Let~$K_0^+(\cat A)$ be the submonoid 
of~$K_0(\cat A)$ generated by the $\alpha_i$, $1\le i\le r$. Since the form $\lr{\cdot,\cdot}$
is non-degenerate
for each~$1\le i\le r$ there exist unique 
$\gamma_{i},\gamma_{-i}\in K_0(\cat A)$ such that for all $1\le j\le r$
$$
\lr{\gamma_{-i},\alpha_j}=q_i^{\delta_{ij}}=\lr{\alpha_j,\gamma_i}
$$
where $q_i=|\End_{\cat A}S_i|$.
Then $\{\gamma_i\}_{1\le i\le r}$ and~$\{\gamma_{-i}\}_{1\le i\le r}$ are bases of~$K_0(\cat A)$.
Let~$c$ be the unique automorphism of~$K_0(\cat A)$ defined by $c(\gamma_{-i})=-\gamma_i$. 

\begin{lemma}\label{lem:gamma}  
\begin{enumerate}[{\rm(a)}]
\item For all~$1\le i\le r$, $\gamma_{\pm i}\in K_0^+(\cat A)$
\item Set 
$$
\Gamma_\pm=\{ \gamma_{\pm i,\pm k}:=c^{\pm k}(\gamma_{\pm i})\,:\, 1\le i\le r,\, k\ge 0\}\cap K_0^+(\cat A).
$$
Then for each $\gamma\in\Gamma_+\cup\Gamma_-$ there exists a unique, up to an isomorphism, indecomposable 
object~$E_\gamma$ with $\gr{E_\gamma}=\gamma$. 
Moreover, an indecomposable~$M$ is preprojective (respectively, preinjective)
if and only if $M\cong E_\gamma$ with $\gamma\in \Gamma_-$ (respectively, $\gamma\in\Gamma_+$).
\item If~$\cat A$ has finitely many isomorphism classes of indecomposables then $\Gamma_+=\Gamma_-$. Otherwise, $\Gamma_+\cap \Gamma_-=\emptyset$ and 
$\gamma_{\pm i,\pm k}\in K_0^+(\cat A)$ for all $1\le i\le r$, $k\ge 0$.
\end{enumerate}
\end{lemma}
\begin{theorem}\label{thm:descr-prepr-Hall} In the notation of Corollary~\ref{cor:hereditary factorization}:
$$H_{{\cat A}_-}=
\overrightarrow{\bigotimes_{\gamma\in\Gamma_-}} \mathbb Q[X_\gamma],\qquad H_{{\cat A}_+}=
\overleftarrow{\bigotimes_{\gamma\in\Gamma_+}} \mathbb Q[X_\gamma]
$$
as vector spaces, while the exponentials $\Exp_{{\cat A}_\pm}$
factor as
$$\Exp_{{\cat A}_-}=\overrightarrow{\prod_{\gamma\in\Gamma_-}}  \Exp_{q_\gamma}(X_\gamma),\qquad 
\Exp_{{\cat A}_+}=\overleftarrow{\prod_{\gamma\in\Gamma_+}}  \Exp_{q_\gamma}(X_\gamma)
$$
where $q_{\gamma_{\pm i,\pm r}}=q_i$, $X_\gamma=[E_\gamma]$ and both products are taken in the total order on~$\Gamma_-$ (respectively, on~$\Gamma_+$) defined
by $\gamma_{-j,-r}<\gamma_{-i,-s}$ (respectively, 
$\gamma_{i,s}<\gamma_{j,r}$) if either $r<s$ or~$r=s$ and $i<j$.
\end{theorem}
We prove Lemma~\ref{lem:gamma}, Theorem~\ref{thm:descr-prepr-Hall} and its Corollary in~\S\ref{sec:Coxeter}.

It turns out that the factorization of~$\Exp_{\cat A}$ allows one to obtain an elementary proof of Ringel's result (cf.~\cite{CX}) that the~$[E_\gamma]$ with 
$\gamma\in\Gamma_+\cup\Gamma_-$ are contained in the composition algebra~$C_{\cat A}$ of~$\cat A$.
After Ringel (\cite{Ringel}), the composition algebra $C_{\cat A}$ of~$\cat A$, that is, the subalgebra of $H_{\cat A}$ generated by the $[S_i]$, $1\le i\le r$,
has the following presentation (cf.~\eqref{eq:Serre})
$$(\ad_{(1,q_i^{},q_i^{2_{}},\dots,q_i^{-a_{ij}})} [S_i])([S_j])=0,\qquad 
(\ad^*_{(1,q_j^{},\dots,q_j^{-a_{ji}})} [S_j])([S_i]),\qquad j<i
$$
(the repeated quantum commutators $(\ad_{(q_0,\dots,q_m)}x)(y)$ and $(\ad^*_{(q_0,\dots,q_m)}x)(y)$ are defined in~\S\ref{sec:adjs}).  
In fact, this algebra is isomorphic to a Sevostyanov's analogue of the upper half of the quantized universal enveloping algebra 
corresponding to the symmetrized Cartan matrix $DA$ (\cite{Sev}).
\begin{corollary}\label{cor:compalg}
For all $\gamma\in \Gamma_+\cup\Gamma_-$, $[E_\gamma]\in C_{\cat A}$.
\end{corollary}
\noindent
This statement is proven in~\S\ref{sec:compalg}.

\subsection*{Acknowledgments}
An important part of this work was done while both authors were visiting the Massachusetts 
Institute of Technology, and it is our pleasure to thank the Department of Mathematics for its hospitality
and Pavel Etingof for his support and stimulating discussions. We are grateful to Jiarui Fei for inspiring conversations.

\section{Preliminaries}

We begin by fixing some notation. Let~$\cat A$ be an abelian category. We will mostly assume that $\cat A$
is a $\kk$-category for some field
$\kk$, that is, there exists an exact functor $\cat A\to \operatorname{Vect}_\kk$, and is essentially small. We denote by $\Iso\cat A$ the set of isomorphism classes of objects in~$\cat A$. Given
an object~$M\in\Ob\cat A$, we denote by~$[M]\in\Iso\cat A$ its isomorphism class. 
For any $\mathcal S\subset \Iso\cat A$, let $\add\mathcal S$ be the smallest additively closed full subcategory of~$\cat A$
containing all direct summands of all objects~$M\in\Ob\cat A$ with $[M]\in \mathcal S$. We set~$\add_{\cat A} M=\add_{\cat A}\{[M]\}$ for any object~$M$ in~$\cat A$.
We denote by $\cat A(\mathcal S)$ the smallest full subcategory of~$\cat A$ containing all objects~$M$ with $[M]\in\mathcal S$ and 
closed under extensions.
We denote by~$K_0^+(\cat A)$ the submonoid of~$K_0(\cat A)$
generated by the $|M|$, $[M]\in\Iso\cat A$ and define a partial order~$\le $ on~$K_0^+(\cat A)$ by $\lambda\le \mu$ if~$\mu-\lambda\in K_0^+(\cat A)$.
If~$\cat B$ is a full subcategory of~$\cat A$ closed under extensions, we denote by~$K_0(\cat B)$ the subgroup of~$K_0(\cat A)$ 
generated by the~$|M|$, $[M]\in\Iso\cat B$, and we set~$K_0^+(\cat B)=K_0^+(\cat A)\cap K_0(\cat B)$.

\subsection{}
Recall that an object $M$ in~$\cat A$ is called {\em indecomposable} if $M$ is not isomorphic to 
a direct sum of two non-zero subobjects.
The category~$\cat A$ is said to be Krull-Schmidt if
every object in~$\cat A$ can be written, uniquely (up to an isomorphism) as a direct sum of finitely many 
indecomposable objects. If~$\cat A$ is Krull-Schmidt, we denote by $\Ind\cat A\subset\Iso\cat A$ the 
set of isomorphism classes of all indecomposable objects in~$\cat A$.
The following simple Lemma will be used repeatedly in the sequel.
\begin{lemma}\label{lem:orth-cond}
Suppose that $\cat A$ is Krull-Schmidt and $\Ind\cat A=\mathcal I_1\bigsqcup\cdots \bigsqcup\mathcal I_r$ such that for all $1\le i<j\le r$, 
$\Hom_{\cat A}(M_i,M_j)=0$ if
$[M_s]\in\mathcal I_s$. Then the categories $\cat A_s=\add_{\cat A}\mathcal I_s$, $1\le s\le r$,  are closed under extensions.
\end{lemma}
\begin{proof}
Consider a short exact sequence 
$$
0\to M\to X\to N\to 0
$$
with $M,N\in\Ob\cat A_s$. Applying $\Hom_{\cat A}(Y,-)$ (respectively, $\Hom_{\cat A}(-,Z)$) with $Y$, $Z$ indecomposable such that
$[Y]\in\mathcal I_t$ (respectively, $[Z]\in\mathcal I_k$) and $t<s<k$, we conclude that $\Hom_{\cat A}(Y,X)=0$ (respectively,
$\Hom_{\cat A}(X,Z)=0$),
hence $X$ cannot have~$Y$ and~$Z$ as its indecomposable summands. 
Thus, the isomorphism classes of
all indecomposable summands of~$X$ lie in~$\mathcal I_s$ and so $X\in\Ob\cat A_s$. 
\end{proof}
If~$\cat A_i$, $i\in I$ are full additive subcategories of $\cat A$ such that the additive monoid $\Iso\cat A$
is the direct sum of its submonoids $\Iso\cat A_i$, we write $\cat A=\bigvee_{i\in I}\cat A_i$. A typical example is provided 
by a partition of the set~$\Ind\cat A$ if $\cat A$ is Krull-Schmidt (and so $\Iso\cat A$ is a 
free monoid generated by $\Ind\cat A$) if~$\Ind \cat A=\bigsqcup_{i\in I} \mathcal S_i$ then 
$\cat A=\bigvee_{i\in I} \add_{\cat A}\mathcal S_i$. Also, 
we denote by~$\cat B_1\cap\cat B_2$ the full subcategory 
of~$\cat A$ whose objects are in~$\cat B_1$ and~$\cat B_2$.

\subsection{}
Given $M,N,K\in\Ob\cat A$, set
$$
\mathcal P^K_{MN}=\{ (f,g)\in\Hom_{\cat A}(N,K)\times \Hom_{\cat A}(K,M)\,:\, \ker g=\operatorname{Im}f\}.
$$
The group $\Aut_{\cat A}(M)\times\Aut_{\cat A}(N)$ acts freely on~$\mathcal P^K_{MN}$ and 
the set of orbits identifies with 
$$
\mathcal F^K_{MN}=\{ U\subset K\,:\, U\cong N,\,K/U\cong M\}.
$$
It is easy to see that isomorphisms of objects $M\to M'$, $N\to N'$ and~$K\to K'$
induce bijections between the sets $\mathcal F^K_{MN}$ and~$\mathcal F^{K'}_{M'K'}$.

Suppose that $\cat A$ is finitary, that is, $\Hom_{\cat A}(M,N)$ and $\Ext^1_{\cat A}(M,N)$ are 
finite sets for all $M,N\in\Ob\cat A$. 
\begin{definition}[Ringel \cite{Ringel}]
The Hall algebra $H_{\cat A}$ of~$\cat A$ is the $\mathbb Q$-vector 
space with the basis $[M]\in\Iso\cat A$ and with the multiplication defined by
$$
[M][N]=\sum_{[K]\in\Iso\cat A} F_{MN}^K [K].
$$
\end{definition}
The number $F_{MN}^K:=\#\mathcal F_{MN}^K$ is independent of representatives of isomorphism classes.

A Hall algebra can be defined for an exact category. In particular, if~$\cat B$ is a full subcategory of~$\cat A$ and is closed under extensions,
then we can consider its Hall algebra~$H_{\cat B}$ which identifies with the subspace of~$H_{\cat A}$ spanned by the $[M]\in\Iso\cat B$.
Note that $H_{\cat B}$ is graded by $K_0(\cat B)$
$$
H_{\cat B}=\bigoplus_{\gamma\in K_0(\cat B)} H_{\cat B,\gamma},
$$
where $H_{\cat B,\gamma}$ is spanned by the~$[M]\in\Iso\cat B$ with~$|M|=\gamma$. Then the completion~$\widehat H_{\cat B}$ of~$H_{\cat B}$
is defined by
$$
\widehat H_{\cat B}=\prod_{\gamma\in K_0(\cat B)} H_{\cat B,\gamma}.
$$

Clearly, there is a natural surjective map $\mathcal P^K_{MN}\to \Ext^1_{\cat A}(M,N)_K$, where 
$\Ext^1_{\cat A}(M,N)_K$ denotes the set of equivalence classes of short exact sequences
$$
0\to N\to E\to M\to 0
$$
with~$E\cong K$. The fiber of that map is given by $\Aut_{\cat A}(K)/\operatorname{Stab}_{\Aut_{\cat A}K}(f,g)$
and it can be shown that the map $\Hom_{\cat A}(M,N)\to \operatorname{Stab}_{\Aut_{\cat A}K}(f,g)$,
$\psi\mapsto 1+f\psi g$ is a bijection. This yields
\begin{proposition}[Riedtmann's formula \cite{Rie}]\label{prop:Rie}
Let~$\cat A$ be a finitary category. Then for all $M,N,K\in\Ob\cat A$
\begin{equation}\label{eq:Rie}
 F_{MN}^K=\frac{|\Ext^1_{\cat A}(M,N)_K||\Aut_{\cat A}(K)|}{|\Aut_{\cat A}(M)\times\Aut_{\cat A}(N)\times
\Hom_{\cat A}(M,N)|} 
\end{equation}
\end{proposition}
We note the following simple corollary which will be used repeatedly.
\begin{corollary}\label{cor:deg-Rie}
\begin{enumerate}[{\rm(a)}]
 \item \label{cor:deg-Rie.a}
Suppose that $\Hom_{\cat A}(N,M)=0=\Ext^1_{\cat A}(M,N)$. Then $[M][N]=[M\oplus N]$.
\item\label{cor:deg-Rie.b} Let $E$ be such that $\End_{\cat A} E$ is a field and $\Ext^1_{\cat A}(E,E)=0$. Then
$[E^{\oplus a}]=\dfrac{[E]^{a}}{[a]_{q_E}!}$ where $q_E=|\End_{\cat A}E|$.
\end{enumerate}
\end{corollary}
\begin{proof}
To prove~\eqref{cor:deg-Rie.a}, observe that
$|\Ext^1_{\cat A}(M,N)_{M\oplus N}|=1$. Then $\Hom_{\cat A}(N,M)=0$ implies that
$$\Aut_{\cat A}(M\oplus N)\cong\Aut_{\cat A}(M)\times\Aut_{\cat A}(N)\times
\Hom_{\cat A}(M,N),$$ 
and it remains to apply~\eqref{eq:Rie}. 

For part~\eqref{cor:deg-Rie.b}, note that $\Aut_{\cat A}E^{\oplus a}\cong \operatorname{GL}(a,\End_{\cat A}E)$,
hence 
$$
\frac{|\Aut_{\cat A}(E^{\oplus (a+1)})|}{|\Aut_{\cat A}(E^{\oplus a})||\Aut_{\cat A}E|}=\frac{\prod_{j=0}^{a} (q_E^{a+1}-q_E^j)}
{(q_E-1)\prod_{j=0}^{a-1} (q_E^a-q_E^j)}=q_E^a [a+1]_{q_E}.
$$
Since $|\Hom_{\cat A}(E^{\oplus a},E)|=q_E^a$, \eqref{eq:Rie} implies that
$$
{[E^{\oplus a}]}{[E]}=[a+1]_{q_E} {[E^{\oplus (a+1)}]},
$$
and the second identity follows by an obvious induction.
\end{proof}
\begin{proposition}\label{prop:decomp}
Let $\cat A$ be a finitary abelian category and
suppose that $\cat A=\cat A_1\bigvee\cdots\bigvee \cat A_r$ where the $\cat A_j$, $1\le j\le r$ are closed under extensions and
for all $1\le i<j\le r$, $\cat A_j\subset \cat A_i^{\rort}$. Then
\begin{enumerate}[{\rm(a)}]
 \item\label{prop:decomp.a} the
multiplication map $m:H_{\cat A_1}\tensor\cdots\tensor H_{\cat A_r}\to H_{\cat A}$ is an isomorphism of vector spaces.
In particular, if~$\cat A=\cat A_1\oplus\cdots \oplus \cat A_r$ then $H_{\cat A}\cong H_{\cat A_1}\tensor\cdots\tensor H_{\cat  A_r}$ as an algebra;

\item \label{prop:decomp.b} If $\Hom_{\cat A}(A_j,A_i)=0$ for all $1\le i<j\le r$, where $A_s\in\Ob\cat A_s$, then $(\cat A_1,\dots,
\cat A_r)$ is a factorizing sequence for~$\cat A$.
\end{enumerate}
\end{proposition}
\begin{proof}
Clearly, it is sufficient to prove this statement for~$r=2$. Let~$\cat B=\cat A_1$, $\cat C=\cat A_2$.
By~\eqref{eq:Rie}, $m([B]\tensor [C])={[B]}{[C]}=c_{[B],[C]}{[B\oplus C]}$
where
$$
c_{[B],[C]}=\frac{|\Aut_{\cat A}(B\oplus C)|}{|\Hom_{\cat A}(B,C)||\Aut_{\cat A}(B)\times\Aut_{\cat A}(C)|}\not=0.
$$
On the other hand, since for any~$[M]\in\Iso\cat A$, there exist a unique $([B_M],[C_M])\in\Iso\cat B\times\Iso\cat C$
with~$[M]=[B_M\oplus C_M]$, the assignment
$[M]\mapsto c_{[B_M],[C_M]}^{-1}[B_M]\tensor [C_M]$
provides a well-defined linear map $\psi:H_{\cat A}\to H_{\cat B}\tensor H_{\cat C}$. Clearly, $\psi\circ m=1_{H_{\cat B}\tensor H_{\cat A}}$
and $m\circ\psi=1_{H_{\cat A}}$. 

If~$\cat A=\cat B\oplus\cat C$ then $\Hom_{\cat A}(B,C)=\Hom_{\cat A}(C,B)=0=\Ext^1_{\cat A}(B,C)=\Ext^1_{\cat A}(C,B)$ for all
$B\in\Ob\cat B$, $C\in\Ob\cat C$. In particular, ${[B]}{[C]}={[C]}{[B]}={[B\oplus C]}$ and so $H_{\cat B}H_{\cat C}=
H_{\cat C}H_{\cat B}$. Thus, in this case $H_{\cat A}\cong H_{\cat B}\tensor H_{\cat C}$ as an algebra.

To prove~\eqref{prop:decomp.b}, note that by Corollary~\ref{cor:deg-Rie}\eqref{cor:deg-Rie.a}, ${[B]}{[C]}={[B\oplus C]}$ for all
$B\in\Ob\cat B$, $C\in\Ob\cat C$. Thus,
\begin{equation*}
\Exp_{\cat B}\Exp_{\cat C}=\sum_{[B]\in\Iso\cat B} [B] \sum_{[C]\in\Iso\cat C} [C]=
\sum_{([B],[C])\in\Iso\cat B\times\Iso\cat C} [B][C]=\sum_{[A]\in\Iso\cat A} [A]=\Exp_{\cat A}.\qedhere
\end{equation*}
\end{proof}

\subsection{}
Given an object $X$ in~$\cat A$ without self-extensions, let $\Exp_{\cat A}([X])=\Exp_{\add_{\cat A}X}$.
For example, if~$\cat A$ is the category of finite dimensional vector spaces over a finite field~$\kk$,
$\Exp_{\cat A}=\Exp_{\cat A}([\kk])$.

\begin{proposition}\label{prop:simpl}
Let $S$ be a simple  object in a finitary abelian category~$\cat A$ with the finite length property. Let~$\cat A'=\add_{\cat A}S$ and let $\cat A''$ be the 
maximal full subcategory of~$\cat A$ whose objects~$M$ satisfy
$[M:S]=0$. 

\begin{enumerate}[{\rm(a)}]
 \item\label{prop:simpl.a} 
 If $S$ is injective, then for any~$M\in\Ob\cat A$, there is a unique subobject~$M''\in\Ob\cat A$ of~$M$ such that
$M/M''\cong S^{\oplus [M:S]}$. In particular, $(\cat A',\cat A'')$ is a factorizing pair. 

\item\label{prop:simpl.b} If $S$ is projective, then for any~$M\in\Ob\cat A$, there is a unique subobject~$M'\cong S^{\oplus [M:S]}$
such that $M/M'\in\Ob\cat A''$. In particular,
$(\cat A'',\cat A')$  is a factorizing pair for~$\cat A$.
\end{enumerate}
\end{proposition}
\begin{proof}
We prove part~\eqref{prop:simpl.a}, the proof for the other part being similar. Let~$M\in\Ob\cat A$. Then
$M$ has a unique maximal subobject~$M''$ which is an object in~$\cat A''$. Indeed, if $X,Y\in\Ob\cat A''$
are subobjects of~$M$, then their sum $X+Y$
which is the image of $X\oplus Y$ in~$M$ under the map canonically induced by the natural inclusions $X\hookrightarrow M$ and
$Y\hookrightarrow M$, is also an object in~$\cat A''$.

We claim that $M/M''\cong S^{\oplus [M:S]}$. 
Indeed, let $a\ge 0$ be maximal such that $M/M''$ has a subobject isomorphic to~$S^{\oplus a}$. Since
$S$ is injective, $M/M''\cong R\oplus S^{\oplus a}$ for some subobject~$R$ of~$M/M''$. Suppose that~$R\not=0$. Then $R$ has a simple 
subobject~$S'$, which by the choice of~$a$ is not isomorphic to~$S$. Let $K$ be the 
subobject of~$M$ such that~$M''$ is a subobject of~$K$ and~$K/M''=S'$. Then $K$ is an object in~$\cat A''$ and
since $K\not=M''$, we obtain a contradiction by the maximality of~$M''$. Thus, $R=0$ and so~$M/M''\cong S^{\oplus a}$. 
Since $[M'':S]=0$,
the short exact sequence $0\to M''\to M\to S^{\oplus a}\to 0$ yields~$a=[M:S]$. 
\end{proof}

\subsection{}
From now on, we only consider abelian $\kk$-categories for a field~$\kk$. In this case, 
if $S$ is a simple object in~$\cat A$, $\End_{\cat A} S$ is a division algebra over~$\kk$. 
Note also that $\Ext^1_{\cat A}(S,S')$ is a right $\End_{\cat A}S$-module and a left $\End_{\cat A}S'$-module.
Consider the $\Ext$-quiver~$\mathcal E_{\cat A}$ of~$\cat A$, which is the valued
quiver with vertices corresponding to isomorphism classes of simple objects and arrows $[S]\xrightarrow{(a,b)} [S']$
if $\Ext^1_{\cat A}(S,S')\not=0$
where $a=\dim_{\End_{\cat A} S'}\Ext_{\cat A}^1(S,S')$ and $b=\dim\Ext_{\cat A}^1(S,S')_{\End_{\cat A}S}$.
We say that~$\cat A$ is acyclic if $\mathcal E_{\cat A}$ is acyclic.

If~$\cat A$ has the finite length property, denote by $[M:S]$ the Jordan-H\"older multiplicity of a 
simple object~$S$ in an object~$M$. 
We say that~$\cat A$ admits a source order if the set of vertices of~$\mathcal E_{\cat A}$ is countable and 
there exists a numbering of vertices such that $[S_i]$ is a source in the quiver obtained by removing all vertices
$[S_j]$, $j<i$ and their adjacent arrows. Equivalently, there exists a numbering of the isomorphism classes of 
simples such that $S_i$ is injective in the Serre subcategory whose objects satisfy $[M:S_j]=0$ if~$j<i$.

\begin{lemma}
Let~$\cat A$ be an acyclic abelian category with the finite length property. Suppose that $\mathcal E_{\cat A}$ has finitely many vertices.
Then $\cat A$ admits a source order.
\end{lemma}
\begin{proof}
Since $\mathcal E_{\cat A}$ is finite and acyclic, it contains at least one source. Let~$[S_1]$ be the corresponding
class and consider the Serre subcategory whose objects satisfy $[M:S_1]=0$. Its $\Ext$ quiver is obtained from that 
of~$\cat A$ by removing the vertex~$[S_1]$ and its adjacent arrows. Continuing this way, we obtain the desired
total order.  
\end{proof}

The following is an immediate consequence of Proposition~\ref{prop:simpl}.
\begin{corollary}\label{cor:sorder}
Let~$\cat A$ be a finitary category and suppose that~$\cat A$ admits a source order~$[S_1]\prec [S_2]\prec\cdots$.
Then 
$$
\Exp_{\cat A}=\overrightarrow{\prod_{i}}\Exp_{\cat A}([S_i])
$$
\end{corollary}

Now we can prove Proposition~\ref{prop:pent.pair} from the introduction.
\begin{proof}[Proof of Proposition~\ref{prop:pent.pair}]
It is immediate from Proposition~\ref{prop:simpl} that $(\cat A_+,\cat A_-)$ is a factorizing pair for~$\cat A$. Clearly,
$\cat A=\cat A_-\bigvee\cat A_0\bigvee\cat A_+$, $\cat A_-\subset \cat A_0^\lort,\cat A_+^\lort$ and
$\cat A_0\subset \cat A_+^\lort$. Suppose that $M$ is an object in~$\cat A_0$. Then $\Hom_{\cat A}(M,S_-)=0$
for any~$[S_-]\in\mathcal S_-$, for otherwise $S_-$ is a direct summand of~$M$, and similarly, $\Hom_{\cat A}(S_+,M)=0$
for all $[S_+]\in\mathcal S_+$. 
Therefore, $(\cat A_-,\cat A_0,\cat A_+)$
is a factorizing triple by Proposition~\ref{prop:decomp}.
\end{proof}

\subsection{}
Let~$\cat B$ be a full subcategory of~$\cat A$ closed under extensions.
Given $\gamma\in K_0(\cat B)$ and an element $X$ of~$H_{\cat B}$ or~$\widehat H_{\cat B}$, 
denote by~$X|_\gamma$ its canonical projection onto~$H_{\cat B,\gamma}$.
Clearly, $X=\sum_{\gamma\in K_0(\cat B)} X|_\gamma$.
Define $E_{\cat B}$ as the subalgebra
of~$H_{\cat B}$ generated by the $\Exp_{\cat B}|_\gamma$, $\gamma\in K_0(\cat B)$.
\begin{lemma}\label{lem:comp-exp}
Suppose that $(\cat A_+,\cat A_-)$ is a factorizing pair for a finitary category~$\cat A$ with the finite length property.
Assume that
\begin{equation}\label{eq:disj-prop}
K_0^+(\cat A_+) \cap K_0^+(\cat A_-)=0.
\end{equation}
Then $E_{\cat A_\pm}\subset E_{\cat A}$.
\end{lemma}
\begin{proof}
Since $(\cat A_+,\cat A_-)$ is a factorizing pair, $\Exp_{\cat A}=\Exp_{\cat A_+}\Exp_{\cat A_-}$ which is 
equivalent to 
$$
\Exp_{\cat A}|_{\gamma}=\sum_{(\gamma_+,\gamma_-)\in K_0^+(\cat A_+)\times K_0^+(\cat A_-)\,:\,\gamma_++\gamma_-=\gamma}(\Exp_{\cat A_+}|_{\gamma_+})(\Exp_{\cat A_-}|_{\gamma_-}),
\qquad \gamma\in K_0^+(\cat A).
$$
Then by~\eqref{eq:disj-prop} we have for all~$\gamma_\pm\in K_0(\cat A_\pm)^+$
$$
\Exp_{\cat A_\pm}|_{\gamma_\pm}=\Exp_{\cat A}|_{\gamma_\pm}-\sum_{(\mu_+,\mu_-)\in(K_0^+(\cat A_+)\setminus\{0\})\times(K_0^+(\cat A_-)\setminus\{0\})\,:\,\mu_++\mu_-=\gamma_\pm}(\Exp_{\cat A_+}|_{\mu_+})(\Exp_{\cat A_-}|_{\mu_-}).
$$
We can now finish the proof by a simultaneous induction on partially ordered sets $K_0^+(\cat A_+)$ and $K_0^+(\cat A_-)$
(note that the induction begins since $|S|\in K_0^+(\cat A_+)\bigsqcup K_0^+(\cat A_-)$ for every simple object~$S$ in~$\cat A$).
\end{proof}

\begin{lemma}\label{lem:comp-exp-1}
Let~$\cat A$ be a finitary abelian category with the finite length property. Let~$\cat B$ be 
a full subcategory of~$\cat A$ closed under extensions and direct summands and satisfying
\begin{enumerate}[$1^\circ$]
 \item If $M,N$ are indecomposable objects in~$\cat B$ with $|M|=|N|$ then $M\cong N$;
\item $H_{\cat B}$ is generated by the $[M]\in \Ind\cat A\cap \Iso\cat B$.
\end{enumerate}
Then $H_{\cat B}=E_{\cat B}$. 
\end{lemma}
\begin{proof}
If~$M\in\Ob\cat B$ is indecomposable then~$1^\circ$ implies that
\begin{equation}\label{eq:exp-ind}
[M]=\Exp_{\cat B}|_{|M|}-\sum_{\text{$[N]$ decomposable}\,:\, |N|=|M|} [N].
\end{equation}
Given~$\gamma\in K_0^+(\cat B)$, define 
$$
H_{\cat B,< \gamma}=\displaystyle\bigoplus_{\mu\in K_0^+(\cat B)\,:\,\mu< \gamma}
H_{\cat B,\mu},\qquad H_{\cat B,\le \gamma}=H_{\cat B,<\gamma}\oplus H_{\cat B,\gamma}.
$$
Clearly, the condition~$2^\circ$ implies that 
if~$M$ is decomposable then
$[M]\in (H_{\cat B,<|M|})^2$.

We prove by induction on the partially ordered set~$K_0^+(\cat B)$ that $H_{\cat B,\le \gamma}\subset E_{\cat B}$
for all~$\gamma\in K_0^+(\cat B)$. 
If~$\gamma\in K_0^+(\cat B)$ is minimal then by~\eqref{eq:exp-ind}
$[M]=\Exp_{\cat B}|_\gamma$ 
with~$[M]$ indecomposable  and so the induction begins. 

Suppose that the claim is proven for all~$\gamma'<\gamma$. Thus, in particular, $H_{\cat B,<\gamma}\subset E_{\cat B}$. 
If~$\gamma$ is such that all objects~$M$ with~$|M|=\gamma$ are decomposable,
we are done by the induction hypothesis. Otherwise, for the unique~$[M]\in\Ind\cat A\cap\Iso\cat B$ with 
$|M|=\gamma$ we have~$[M]\in E_{\cat B}$ by~\eqref{eq:exp-ind} and the induction hypothesis, hence 
$H_{\cat B,\gamma}\subset E_{\cat B}$ and so~$H_{\cat B,\le \gamma}\subset E_{\cat B}$.
\end{proof}

\section{\texorpdfstring{$q$}{q}-commutators and proof of Theorem~\ref{th:local conjugation}}\label{sec:comm}

\subsection{}
Given an integral domain~$R$ and $q\in R$, $q\not=0$, define an $R$-linear operator $D_{q}$ on~$R[[x]]$ (the $q$-derivative) by
$$
(D_{q}f)(x)=\frac{f(qx)-f(x)}{(q-1)x}
$$
and set 
$$
D^{(j)}_{q}:=\frac{1}{[j]_q!}\, D^j_{q}.
$$
Since, clearly
\begin{equation}\label{eq:q-der}
D^{(j)}_{q}x^a=\qbinom{a}{j} x^{a-j},
\end{equation}
where 
$$
\qbinom{a}{j}=\frac{[a]_q!}{[j]_q! [a-j]_q!}
$$
is the Guassian $q$-binomial which is well-known to be a polynomial in~$q$,
the $D^{(j)}_q$ are well-defined $R$-endomorphisms of~$R[[x]]$ and of~$R[x]$ regarded as a subring of~$R[[x]]$.
We gather some elementary properties of~$D_q$ in the following Lemma.
\begin{lemma}\label{lem:q-leibn}
\begin{enumerate}[{\rm(a)}]
 \item\label{lem:q-leibn.a} 
For all $f,g\in R[[x]]$,
$$
D^{(j)}_q(f g)=\sum_{s=0}^j (D^{(j-s)}_q f)(q^s x) D^{(s)}_q g.
$$
\item\label{lem:q-leibn.b} If $f\in R[x]$ is invertible in~$R[[x]]$ then 
$$
\prod_{s=0}^j f(q^s x) D^{(j)}_q(f^{-1})\in R[x],\qquad j\ge 0.
$$
\end{enumerate}
\end{lemma}
\begin{proof}
For $j=1$ it is easy to check that
\begin{equation}\label{eq:q-leibn-1}
D_q(fg)=f(qx)D_q(g)+g (D_q f).
\end{equation}
For the inductive step, we have
$$
D_q^{(r+1)}(f g)=\frac{1}{[r+1]_q}\sum_{s=0}^r D_q((D_q^{(r-s)}f)(q^s x) D^{(s)}_q g).
$$
Since $D_q (P(q^a x))=q^a (D_q P)(q^a x)$, we obtain
$$
D_q^{(r+1)}(fg)=\frac{1}{[r+1]_q}\sum_{s=0}^r q^s [r+1-s]_q (D_q^{(r+1-s)}f)(q^s x) D^{(s)}_q g+
[s+1]_q (D_q^{(r-s)}f)(q^{s+1} x)D^{(s+1)}_q g.
$$
To complete part~\eqref{lem:q-leibn.a}, it remains to observe that $q^s [r+1-s]_q+[s]_q=[r+1]_q$.

For part~\eqref{lem:q-leibn.b}, we use induction on~$j$. Note that the assertion trivially holds for~$j=0$ and~\eqref{lem:q-leibn.a} implies
$$
\sum_{j=0}^{r} (D_q^{(r-j)} f)(q^j x)D_q^{(j)}(f^{-1})=0.
$$
Therefore,
$$
\Big(\prod_{j=0}^{r+1} f(q^j x)\Big)D_q^{(r+1)}(f^{-1})=-\sum_{j=0}^r (D_q^{(r-j)}f)(q^j x)\Big(\prod_{t=0}^r f(q^t x)\Big) D_q^{(j)}(f^{-1}).
$$
By the induction hypothesis, $\Big(\prod_{t=0}^j f(q^t x)\Big)D_q^{(j)}(f^{-1})\in R[x]$, hence the right hand side 
of the above expression is contained in~$R[x]$. 
\end{proof}

\subsection{}\label{sec:adjs}
Let $F$ be a field and let~$\mathcal A$ be a unital associative $F$-algebra. Given $x\in\mathcal A$ and~$q\in F^\times$, let~$\ad_q x$, $\ad_q^* x$ be the $F$-linear
endomorphisms of~$\mathcal A$ defined by $\ad_q x(y)=xy-q yx$ and $\ad_q^* x(y)=yx-q xy$ for all~$y\in\mathcal A$.
Since, clearly $\ad_q x \circ \ad_{q'}x=\ad_{q'}x\circ \ad_q x$, we can define $\ad_{(q_0,\dots,q_r)}x=\prod_{i=0}^r \ad_{q_i}x$
and $\ad^*_{(q_0,\dots,q_r)}x=\prod_{i=0}^r \ad_{q_i}^* x$.
\begin{lemma}\label{lem:seq}
Let $x,y\in\mathcal A$ and $(q_0,\dots,q_r)\in (F^\times)^{r+1}$. Then
\begin{equation}\label{eq:key-id-a}
\begin{split}
&(\ad_{(q_0,\dots,q_r)}x)(y)=\sum_{j=0}^{r+1} (-1)^j e_j(q_0,\dots,q_r) x^{r+1-j}y x^j,\\
&(\ad_{(q_0,\dots,q_r)}^* x)(y)=\sum_{j=0}^{r+1} (-1)^j e_j(q_0,\dots,q_r) x^{j}y x^{r+1-j},
\end{split}
\end{equation}
and
\begin{equation}\label{eq:key-id-0}
\begin{split}
&x^r y=\sum_{j=0}^r h_{r-j}(q_0,\dots,q_j) (\ad_{(q_0,\dots,q_{j-1})}x)(y) x^{r-j},\\
&yx^r=\sum_{j=0}^r h_{r-j}(q_0,\dots,q_j) x^{r-j} (\ad_{(q_0,\dots,q_{j-1})}^*x)(y),
\end{split}
\end{equation}
where $e_s$ (respectively, $h_s$) denotes the elementary (respectively, the complete) symmetric polynomial of degree~$s$. In particular,
if $q_0,q\in F^\times$
\begin{equation}\label{eq:key-id-b}
\begin{split}
&(\ad_{q_0(1,\dots,q^r)}x)(y)=\sum_{j=0}^r (-1)^j q_0^j q^{\binom{j}2}\qbinom{r+1}{j} x^{r+1-j}y x^j,
\\
&(\ad_{q_0(1,\dots,q^r)}^*x)(y)=\sum_{j=0}^r (-1)^j q_0^j q^{\binom{j}2}\qbinom{r+1}{j} x^{j}y x^{r+1-j}
\end{split}
\end{equation}
and
\begin{equation}\label{eq:key-id}
x^r y=\sum_{j=0}^r q_0^{r-j} \qbinom{r}{j} (\ad_{q_0(1,\dots,q^{j-1})}x)(y) x^{r-j},\qquad 
yx^r=\sum_{j=0}^r q_0^{r-j}\qbinom{r}{j} x^{r-j} (\ad_{q_0(1,\dots,q^{j-1})}^*x)(y).
\end{equation}
\end{lemma}
\begin{proof}
We prove the identities involving~$\ad$, the proof of the ones involving~$\ad^*$ being similar. The argument is by
induction on~$r$, the case~$r=0$ being obvious in both~\eqref{eq:key-id-a} and~\eqref{eq:key-id-0}. 
For the inductive step in~\eqref{eq:key-id-a}, we have
\begin{align*}
(\ad_{(q_0,\dots,q_r)}x)(y)&=\sum_{j=0}^r (-1)^j e_j(q_0,\dots,q_{r-1})x^{r+1-j}y x^j-\sum_{j=0}^r (-1)^j e_j(q_0,\dots,q_{r-1})q_r x^{r-j}y x^{j+1}\\
&=\sum_{j=0}^{r+1} (-1)^j (e_j(q_0,\dots,q_{r-1})+e_{j-1}(q_0,\dots,q_{r-1})q_r)x^{r+1-j}y x^j
\end{align*}
where we use the usual convention that~$e_s(x_1,\dots,x_k)=0$ if~$s<0$ or~$s>k$.
To prove~\eqref{eq:key-id-0}, 
let~$y_0=y$ and define $y_{j+1}=(\ad_{q_j} x)(y_j)$, $j\ge 0$. 
Then
\begin{align*}
x^{r+1}y&=\sum_{j=0}^r h_{r-j}(q_0,\dots,q_j) (q_j y_j x^{r+1-j}+y_{j+1}x^{r-j})
\\&=
\sum_{j=0}^{r+1} (h_{r-j}(q_0,\dots,q_j)q_j+h_{r+1-j}(q_0,\dots,q_{j-1}))y_jx^{r+1-j}. 
\end{align*}
To complete both inductive steps, it remains to observe that 
$$
e_j(q_0,\dots,q_{r-1})+e_{j-1}(q_0,\dots,q_{r-1})q_r=e_j(q_0,\dots,q_r)
$$
and
$$h_{r-j}(q_0,\dots,q_j)q_j+h_{r+1-j}(q_0,\dots,q_{j-1})
=h_{r+1-j}(q_0,\dots,q_j)$$ 
which follow from the formulae
$$
\sum_{s=0}^{k+1} e_s(x_0,\dots,x_k)t^s=\prod_{i=0}^k (1+x_i t),\qquad \sum_{s\ge 0} h_s(x_0,\dots,x_k)t^s=\prod_{j=0}^k \frac{1}{1-x_j t}.
$$
Since 
\begin{equation}\label{eq:well-known}
\prod_{r=0}^n (1+q^r t)=\sum_{s\ge 0} q^{\binom{s}2}\qbinom{n+1}{s}t^s,\qquad
\prod_{r=0}^{n} \frac1{1-q^r t}=\sum_{s\ge 0}\qbinom{s+n}{s}t^s,
\end{equation}
\eqref{eq:key-id-b} and~\eqref{eq:key-id} follow, respectively, from~\eqref{eq:key-id-a} and~\eqref{eq:key-id-0}
\end{proof}
Suppose that $\mathcal A$ admits a completion $\widehat{\mathcal A}$. 
\begin{corollary}\label{cor:seq}
Let $x,y\in\mathcal A$ and
suppose that $(\ad_{q_0(1,q,\dots,q^{r})}x)(y)=0$ for some~$q_0,q\in F^\times$ and~$r\ge 0$ and 
that the assignment $t\mapsto x$ extends to an algebra homomorphism $F[[t]]\to \widehat{\mathcal A}$. Then for any $P\in F[[t]]$ we have in  
$\widehat{\mathcal A}$
$$
P(x)y=\sum_{j=0}^r (\ad_{q_0(1,q,\dots,q^{j-1})} x)(y) (D^{(j)}_{q}P)(q_0 x).
$$
Similarly, if $(\ad^*_{q_0(1,q,\dots,q^{s})}x)(y)=0$, 
$$
yP(x)=\sum_{j=0}^s (D^{(j)}_{q}P)(q_0 x)(\ad^*_{q_0(1,q,\dots,q^{j-1})} x)(y) 
$$
\end{corollary}

Given $q\in F^\times$, define
$$
\exp_q(t)=\sum_{j\ge 0} \frac{t^j}{[j]_q!}\in F[[t]].
$$
\begin{lemma}
In $F [[t]]$, we have 
$D_q(\exp_q(t))=\exp_q(t)$ and for $\nu\in\mathbb Z$
$$
\exp_q(t)^{-1}\exp_q(q^\nu t)=:\Phi_{\nu}(t,q)
=\begin{cases}\displaystyle\prod_{r=0}^{\nu-1}(1+q^r(q-1)t),&\nu\ge 0\\
                               \displaystyle\prod_{r=1}^{-\nu}(1+q^{-r}(q-1)t)^{-1},&\nu<0.
                              \end{cases}
$$
\end{lemma}
\begin{proof}
The first identity is obvious. 
Since 
$$
\exp_q(t)^{-1}\exp_q(q^\nu t)=\sum_{a,b\ge 0}\frac{(-1)^a q^{\nu b}t^{a+b}}{[a]_q![b]_q!}=
\sum_{r\ge 0} \Big(\sum_{a=0}^r (-1)^a q^{\binom{a}2+\nu(r- a)}\qbinom{r}{a}\Big) \frac{t^r}{[r]_q!}
$$
and 
by~\eqref{eq:well-known} the inner sum equals $$
q^{\nu r}\,\prod_{a=0}^{r-1} (1-q^{a-\nu}).$$
Suppose first that~$\nu\ge0$. Then we can rewrite it as
$$ 
\prod_{a=0}^{r-1}(q^{\nu}-q^a)=[r]_q!q^{\binom{r}2}(q-1)^r\qbinom{\nu}{r}
$$ 
hence 
$$
\exp_q(t)^{-1}\exp_q(q^\nu t)=\sum_{r\ge 0} q^{\binom{r}2}\qbinom{\nu}{r}((q-1)t)^r
$$
which by~\eqref{eq:well-known} equals $\prod_{r=0}^{\nu-1} (1+q^r (q-1)t)$.

If~$\nu<0$, we write
$$
q^{\nu r}\,\prod_{a=0}^{r-1}(1-q^{-\nu+a})=(-1)^r q^{\nu r}(q-1)^r\prod_{a=0}^{r-1} [-\nu+r-1-a]_q=(-1)^r q^{\nu r}[r]_q!(q-1)^r \qbinom{-\nu+r-1}{r}
$$
hence
$$
\exp_q(t)^{-1}\exp_q(q^\nu t)=\sum_{r\ge 0}\qbinom{-\nu+r-1}{r} (-(q-1)q^\nu t)^r=\prod_{s=0}^{-\nu-1} (1+(q-1)q^{\nu+s}t)^{-1}
$$
where we used the second identity in~\eqref{eq:well-known}.
\end{proof}
\begin{corollary}\label{cor:act-conj}
Let~$x,y\in\mathcal A$ and assume that $(\ad_{(q^{\nu},q^{\nu+1},\dots,q^{\nu+r})}x)(y)=0$ for some~$q\in F^\times$, $\nu\in\mathbb Z$ and~$r\ge 0$. Then in~$\widehat{\mathcal A}$
$$
\exp_q(x)y\exp_q(x)^{-1}=
\Big(\sum_{j=0}^{r} \frac{1}{[j]_{q}!}\,(\ad_{(q^\nu,\dots,q^{\nu+j-1})}x)(y)\Big)\Phi_\nu(x,q)
$$
and
$$
\exp_q(x)^{-1}y\exp_q(x)=
\sum_{j=0}^{r} \frac{(-1)^j q^{\binom{j}2}}{[j]_{q}!}\,(\ad_{(q^{\nu},\dots,q^{\nu+j-1})}x)(y)\Phi_{-j-\nu}(q^{\nu+j}x,q).
$$
In particular, $\exp(x)y\exp(x)^{-1}\in\mathcal A$ (respectively, $\exp(x)^{-1}y\exp(x)\in\mathcal A$) if and only 
if~$\nu\ge 0$ (respectively, $-\nu\ge r$). 
Similarly, if $(\ad_{(q^{\mu},q^{\mu+1},\dots,q^{\mu+s})}^*x)(y)=0$ for some~$q\in R$, $\mu\in\mathbb Z$ and $s\ge 0$
$$
\exp_q(x)^{-1} y\exp_q(x)=
\Phi_{\mu}(x,q)
\Big(\sum_{j=0}^{s} \frac{1}{[j]_{q}!}\,(\ad_{(q^\mu,\dots,q^{\mu+j-1})}^*x)(y)\Big)
$$
and
$$
\exp_q(x)y\exp_q(x)^{-1}=
\sum_{j=0}^{s} \frac{(-1)^j q^{\binom{j}2}}{[j]_{q}!}\,\Phi_{-j-\mu}(q^{j+\mu}x,q)(\ad_{(q^{\mu},\dots,q^{\mu+j-1})}^*x)(y).
$$
In particular, $\exp_q(x)^{-1}y\exp_q(x)\in\mathcal A$ (respectively, $\exp_q(x) y\exp_q(x)^{-1}\in\mathcal A$) if and only 
if~$\mu\ge 0$ (respectively, $-\mu \ge s$).
\end{corollary}

Recall that a multiplicatively closed subset $\mathcal S$ of an algebra~$\mathcal A$ is said to be an Ore set if~$\mathcal S$
does not contain zero divisors and 
for all $a\in\mathcal A$, $s\in\mathcal S$ there exist $u,u'\in\mathcal A$ and $t,t'\in \mathcal S$ such that 
$at=su$ and $t'a=u's$.
\begin{proposition}\label{prop:abs-Ore}
Let~$\mathcal Y$ be a generating set for~$\mathcal A$. Let~$x\in\mathcal A$ be such that the assignment 
$t\mapsto x$ extends to an algebra homomorphism $F[[t]]\mapsto \widehat{\mathcal A}$ and suppose that there exists $q\in F^\times$ such that 
for all $y\in\mathcal Y$ either $(\ad_{(q^{\nu},q^{\nu+1},\dots,q^{\nu+r})}x)(y)=0$ or
$(\ad_{(q^{\nu},q^{\nu+1},\dots,q^{\nu+r})}^*x)(y)=0$ for some $\nu=\nu(y)\in\mathbb Z$, $r=r(y)\ge 0$. 
Let~$\mathcal S_x$ be the minimal multiplicatively closed subset of~$\mathcal A$
containing $1+q^s(q-1)x$ for all~$s\in\mathbb Z$. Then $\mathcal S_x$ is an Ore set in~$\mathcal A$ and 
the assignment $z\mapsto \exp_q(x)z\exp_q(x)^{-1}$, $z\in\mathcal A$ defines an automorphism of
$\mathcal A[\mathcal S_x^{-1}]$. 
\end{proposition}
\begin{proof}
The elements of~$\mathcal S_x$ are not zero divisors since they are invertible in the completion~$\widehat{\mathcal A}$.
Since the elements of~$\mathcal S_x$ commute, it is enough to prove that the Ore condition
holds for all $y\in\mathcal A$ and $s\in\mathcal S_x$.

We need the following simple 
\begin{lemma}\label{lem:tmp-1}
Let $P\in F[t]$. Then $P$ divides $D_q^{(s)}\big(\prod_{j=0}^r P(q^{-j} t)\big)$ for all $0\le s\le r$. 
\end{lemma}
\begin{proof}
Let $P_r=\prod_{j=0}^r P(q^{-j}t)$.
The argument is by induction on~$r$. For~$r=0$ there is nothing to do. For the inductive step, note that
by Lemma~\ref{lem:q-leibn}\eqref{lem:q-leibn.a}
$$
D_q^{(s)}P_{r+1}=\sum_{a=0}^s q^{-(r+1)(s-a)}(D_q^{(s-a)}P)(q^{a-r-1}t)D_q^{(a)} P_r.
$$
If~$s<r+1$ then  by the induction hypothesis $P$ divides $D_q^{(a)}P_r$, $0\le a\le s$. If~$s=r+1$,
the only term in the above sum to which the induction hypothesis does not apply is that with~$a=r+1=s$,
which equals $P(t)D_q^{(r+1)}P_r$.
\end{proof}

Suppose that~$y\in\mathcal A$ satisfies $(\ad_{(q^{\nu},q^{\nu+1},\dots,q^{\nu+r})}^*x)(y)=0$ for some~$q\in R^\times$ and
$r,\nu\ge0$ (the case of $(\ad_{(q^\nu,\dots,q^{\mu+r})}x)(y)=0$ is similar).
Let $P\in \mathcal S_x$ and set
$$
P'=\prod_{j=0}^r P(q^{-j-\nu}x)
$$
Clearly, $P'\in\mathcal S_x$. Then by Corollary~\ref{cor:seq}
$$
y P'=\sum_{j=0}^r q^{-j\nu} D_{q}^{(j)}\Big(\prod_{t=0}^{r} P(q^{-t}x)\Big)Y_j=PY,\qquad Y_j,Y\in \mathcal A,
$$
since, by the above Lemma, $P$ divides
$
D_{q}^{(j)}\Big(\prod_{t=0}^{k} P(q^{-t}x)\Big)
$
for all~$0\le j\le k$.

Since $P$ is invertible in~$\widehat{\mathcal A}$, we further have
$$
y P^{-1}=\sum_{j=0}^r (D^{(j)}_{q} P^{-1})(q^{\nu }x)Y_j.
$$
Let $P''=\prod_{s=0}^r P(q^{\nu+s}x)$.
By Lemma~\ref{lem:q-leibn}\eqref{lem:q-leibn.b}, $P'' (D^{(j)}_{q} P^{-1})$ is a polynomial in~$x$ for all~$0\le j\le k$.
Therefore, $P''y P^{-1}\in \mathcal A$. 

Thus, $\mathcal S_x$ is an Ore set in~$\mathcal A$. It remains to apply Corollary~\ref{cor:act-conj}.
\end{proof}

\subsection{}
We will now use the above identities to study the action of certain elements of the quantum Chevalley group on the Hall algebra.

Recall that an object~$E$ in an abelian category~$\cat A$ is called a {\em brick} if~$\End_{\cat A}E$ is 
a division ring. 
In that case, $\Ext^i_{\cat A}(E,M)$ is a right $\End_{\cat A}E$-vector space while $\Ext^i_{\cat A}(M,E)$ is 
a left $\End_{\cat A}E$-vector space for all~$i\ge 0$. Set
$$
\nu_E(M)=\dim \Hom_{\cat A}(E,M)_{\End_{\cat A}E}-
\dim_{\End_{\cat A}E}\Hom_{\cat A}(M,E).
$$
We say that a brick~$E$ is {\em exceptional} if $\Ext^1_{\cat A}(E,E)=0$.

\begin{lemma}\label{lem:brick1}
Let~$E$ be an exceptional brick in an abelian category~$\cat A$ and
let~$M$ be an object in~$\cat A$.
If $\Ext^1_{\cat A}(E,M)=0$
then for a non-split short exact sequence
$$
0\to E\to U\to M\to 0
$$
we have $\Hom_{\cat A}(U,E)\cong \Hom_{\cat A}(M,E)$, $\Ext^1_{\cat A}(E,U)=0$,
$\Ext^1_{\cat A}(U,E)\cong\Ext^1_{\cat A}(M,E)/\End_{\cat A}E$ and
$\Hom_{\cat A}(E,M)\cong\Hom_{\cat A}(E,U)/\End_{\cat A}E$. Similarly,
if~$\Ext^1_{\cat A}(M,E)=0$ then for a non-split short exact sequence 
$$
0\to M\to U\to E\to 0
$$
we have 
$\Hom_{\cat A}(E,U)\cong \Hom_{\cat A}(E,M)$, $\Ext^1_{\cat A}(U,E)=0$,
$\Ext^1_{\cat A}(E,U)\cong \Ext^1_{\cat A}(E,M)/\End_{\cat A}E$ and
$\Hom_{\cat A}(M,E)\cong\Hom_{\cat A}(U,E)/\End_{\cat A}E$.
\end{lemma}
\begin{proof}
We prove the first statement only, the argument for the second one being similar. Consider 
a non-split short exact sequence 
\begin{equation}\label{eq:tmp}
0\to E\xrightarrow{\iota} U\to M\to 0.
\end{equation}
Applying~$\Hom_{\cat A}(-,E)$ 
we obtain 
$$
0\to\Hom_{\cat A}(M,E)\to\Hom_{\cat A}(U,E)\to\End_{\cat A} E\to\Ext^1_{\cat A}(M,E)\to\Ext^1_{\cat A}(U,E)\to 0.
$$
We claim that the natural morphism $\Hom_{\cat A}(U,E)\to \End_{\cat A}E$, $f\mapsto f\circ\iota$ is identically zero. Otherwise, 
since it is a morphism of left $\End_{\cat A}E$-vector spaces and $\End_{\cat A}E$ is one dimensional as such, it is surjective and so
there exists $f\in\Hom_{\cat A}(U,E)$ such that $f\circ\iota=1_E$ hence~\eqref{eq:tmp} splits. Thus,
$\Hom_{\cat A}(U,E)\cong \Hom_{\cat A}(M,E)$ and we have a short exact sequence 
$$
0\to\End_{\cat A}E\to\Ext^1_{\cat A}(M,E)\to\Ext^1_{\cat A}(U,E)\to 0.
$$ 

On the other hand, applying $\Hom_{\cat A}(E,-)$ to~\eqref{eq:tmp} we obtain a long exact sequence
\begin{equation*}
0\to \End_{\cat A}E\to \Hom_{\cat A}(E,U)\to\Hom_{\cat A}(E,M)\to 0\to \Ext^1_{\cat A}(E,U)\to\Ext^1_{\cat A}(E,M)\to 0
\end{equation*}
which yields the remaining isomorphisms.
\end{proof}

\subsection{}
Let~$\cat A$ be a finitary abelian category. Let~$E$ be an exceptional brick in~$\cat A$ and 
denote $\kk_E=\End_{\cat A}E$ and $q_E=|\kk_E|$.
\begin{proposition}\label{prop:ore-set}
If $\Ext^1_{\cat A}(E,M)=0$, then 
$$
(\ad^*_{(q_E^{\nu_E(M)},\dots,q_E^{\nu_E(M)+r})}{[E]})({[M]})=0,
$$
where~$r=\dim_{\kk_E}\Ext^1_{\cat A}(M,E)$, and for any $P\in\mathbb Q[[x]]$ we have
in $\widehat H_{\cat A}$
$$
{[M]}P({[E]})=\sum_{j= 0}^r (D^{(j)}_{q_E}P)(q_E^{\nu_E(M)} {[E]})(\ad^*_{(q_E^{\nu_E(M)},\dots,q_E^{\nu_E(M)+j-1})}{[E]})({[M]})
$$
Similarly, if $\Ext^1_{\cat A}(M,E)=0$,
$$
(\ad_{(q_E^{-\nu_E(M)},\dots,q_E^{-\nu_E(M)+s})}{[E]})({[M]})=0,
$$
where $s=\dim_{\kk_E}\Ext^1_{\cat A}(E,M)$, and for any~$P\in\QQ[[x]]$ we have in~$\widehat H_{\cat A}$
$$
P({[E]}){[M]}=\sum_{j\ge 0} (\ad_{(q_E^{-\nu_E(M)},\dots,q_E^{-\nu_E(M)+j-1})}{[E]})({[M]}) (D^{(j)}_{q_E}P)(q_E^{-\nu_E(M)}{[E]}).
$$

\end{proposition}
\begin{proof}
We prove only the first statement, the prove of the second one being similar.
We will need the following Lemma
\begin{lemma}
If~$\Ext^1_{\cat A}(E,M)=0$,
$$
(\ad^*_{(q_E^{\nu_E(M)},\dots,q_E^{\nu_E(M)+j-1})}{[E]})({[M]})=\sum_{[U]\in \Iso\cat A} c_{[U]}[U]
$$
where $c_{[U]}\not=0$ implies that 
\begin{equation}\label{eq:cond-1}
\dim_{\kk_E}\Ext^1_{\cat A}(U,E)=\dim_{\kk_E}\Ext^1(M,E)-j,\qquad \nu_E(U)=\nu_E(M)+j.
\end{equation}
In particular, 
$$
(\ad^*_{(q_E^{\nu_E(M)},\dots,q_E^{\nu_E(M)+j})}{[E]})({[M]})=0,\qquad j\ge\dim_{\kk_E}\Ext^1(M,E).
$$
Similarly, if~$\Ext^1_{\cat A}(M,E)=0$,
$$
(\ad_{(q_E^{-\nu_E(M)},\dots,q_E^{-\nu_E(M)+j-1})}{[E]})({[M]})=\sum_{[U']\in\Iso\cat A} c_{[U']} {[U']}
$$
where $c_{[U']}=0$ unless
\begin{equation}\label{eq:cond-1a}
\dim_{\kk_E}\Ext^1_{\cat A}(E,U')=\dim_{\kk_E}\Ext^1(E,U')-j,\qquad \nu_E(U')=\nu_E(M)-j.
\end{equation}
In particular, 
$$
(\ad_{(q_E^{-\nu_E(M)},\dots,q_E^{-\nu_E(M)+j})}{[E]})({[M]})=0,\qquad j\ge\dim_{\kk_E}\Ext^1(E,M).
$$
\end{lemma}
\begin{proof}
Consider~$U$ such that $\Ext^1_{\cat A}(E,U)=0$. We have
by~\eqref{eq:Rie}
$$
{[E]}{[U]}=\frac{|\Aut_{\cat A}(U\oplus E)|}
{|\Hom_{\cat A}(E,U)||\Aut_{\cat A}U||\kk_E^\times|}\,{[U\oplus E]}.
$$
On the other hand, by Lemma~\ref{lem:brick1} and~\eqref{eq:Rie}
$$
{[U]}{[E]}=\frac{|\Aut_{\cat A}(U\oplus E)|}
{|\Hom_{\cat A}(U,E)||\Aut_{\cat A}U||\kk_E^\times|}\,{[U\oplus E]}+
\sum_{[U']\in\Iso\cat A} c_{[U']}[U']=q_E^{\nu_E(U)} {[E]}{[U]}+\sum_{[U']\in\Iso\cat A} c_{[U']}[U']
$$
where $c_{[U']}\not=0$ implies that we have a non-split short exact sequence
$$
0\to E\to U'\to U\to 0.
$$
By Lemma~\ref{lem:brick1}, 
$\nu_E(U')=\nu_E(U)+1$ and~$\dim_{\kk_E}\Ext^1_{\cat A}(U',E)=\dim_{\kk_E}\Ext^1_{\cat A}(U,E)-1$.
An obvious induction now completes the proof of the Lemma.
\end{proof}

To complete the proof of the proposition, it remains to apply the Lemma, together with Corollary~\ref{cor:seq} 
with~$x={[E]}$, $y={[M]}$, $q_0=q_E^{\nu_E(M)}$, $q=q_E$ and~$r=\dim_{\kk_E}\Ext^1_{\cat A}(M,E)$ (respectively,
$s=\dim_{\kk_E}\Ext^1_{\cat A}(E,M)$).
\end{proof}
Thus, if $E$ is an exceptional brick and $\Ext^1_{\cat A}(E,M)=0$
\begin{equation}\label{eq:fund-rel}
\sum_{j=0}^{r+1} (-1)^j q_E^{\binom{j}2+j\nu_E(M)}\qbinom[q_E]{r+1}{j} {[E]}^j {[M]}{[E]}^{r+1-j}=0,\qquad r=\dim_{\kk_E}\Ext^1_{\cat A}(M,E)
\end{equation}
while if $\Ext^1_{\cat A}(M,E)=0$,
\begin{equation}\label{eq:fund-rel'}
\sum_{j=0}^{r+1} (-1)^j q_E^{\binom{j}2-j\nu_E(M)}\qbinom[q_E]{r+1}{j} {[E]}^{r+1-j} {[M]}{[E]}^{j}=0,\qquad r=\dim_{\kk_E}\Ext^1_{\cat A}(E,M).
\end{equation}
In the special case when $(E,E')$ is an orthogonal exceptional pair of bricks (that is, $
\Hom_{\cat A}(E,E')=\Hom_{\cat A}(E',E)=\Ext^1_{\cat A}(E,E')=0$), we obtain the so-called fundamental relations in the 
Hall algebra of the
subcategory $\cat A(E,E')$ of~$\cat A$ which is defined to be the smallest full subcategory of~$\cat A$
containing $E,E'$ and closed under extensions. In particular, this yields Serre relations in~$C_{\cat A}$.
Namely, $S$ and~$S'$ are non-isomorphic simples with $\Ext^1_{\cat A}(S',S)=0$, $\dim_{\kk_S}\Ext^1_{\cat A}(S,S')=r$ and
$\dim_{\kk_{S'}}\Ext^1_{\cat A}(S,S')=r'$ then
\begin{equation}\label{eq:Serre}
\sum_{j=0}^{r'+1} (-1)^j q_{S'}^{\binom{j}2}\qbinom[q_{S'}]{r'+1}{j} {[S']}^j {[S]}{[S']}^{r'+1-j}=0=
\sum_{j=0}^{r+1} (-1)^j q_S^{\binom{j}2}\qbinom[q_S]{r+1}{j} {[S]}^{r+1-j} {[S']}{[S]}^{j}
\end{equation}

Propositions~\ref{prop:abs-Ore} and~\ref{prop:ore-set}
immediately yield 
\begin{theorem}\label{prop:auto}
Let $E$ be an exceptional brick and let $\mathcal S_E$ be the minimal multiplicatively closed 
subset of~$H_{\cat A}$ containing~$1+(q_E-1)q_E^j{[E]}$ for all $j\in\mathbb Z$. 
Then
$\mathcal S_E$ is an Ore set in~$H_{\add_{\cat A}(E)^{\lort}}$ and in $H_{\add_{\cat A}(E)^{\rort}}$
and the assignment $[M]\mapsto \mathcal \Exp_{\cat A}([E])[M]\Exp_{\cat A}([E])^{-1}$
extends to an automorphism of $H_{\add_{\cat A}(E)^{\lort}}[\mathcal S_E^{-1}]$ and $H_{\add_{\cat A}(E)^{\rort}}[\mathcal S_E^{-1}]$.
In particular, if $E$ is a projective or an injective indecomposable, this assignment defines an automorphism of
$H_{\cat A}[\mathcal S_E^{-1}]$.
\end{theorem}
Therefore, Theorem~\ref{th:local conjugation} is proved.

\section{Preprojective and preinjective factorization}

\subsection{}
Let $\cat A$ be a hereditary finitary category with the finite length property and with a projective generator and 
an injective co-generator. We may assume, up to equivalence, that $\cat A$ is the category $\Mod\Lambda$ of finitely generated left
modules over a hereditary Artin $\kk$-algebra~$\Lambda$.
We now gather the properties of preprojective and preinjective objects in the category $\cat A$ that will be needed later.

Let $\tau^+:
\cat A\to\cat A$ be
the Auslander-Reiten translation and $\tau^-$ 
be its left adjoint functor.
By~\cite{ARS}*{Corollary~IV.4.7}, for any objects $X,Y\in\Ob\cat A$ we have
\begin{equation}\label{eq:Ext-Hom}
\dim_{\kk}\Ext^1_{\cat A}(X,Y)=\dim_{\kk}\Hom_{\cat A}(Y,\tau^+ X).
\end{equation}

An indecomposable~$M\in\Ob\cat A$
is called {\em preprojective} (respectively, {\em preinjective}) if for some non-negative integer~$n$, $(\tau^+)^n M$ is a non-zero projective
(respectively, $(\tau^-)^n M$ is a non-zero injective) object. 
Let $\widetilde{\cat P}(\cat A)$ (respectively, $\widetilde{\cat I}(\cat A)$)
be the set of isomorphism classes of indecomposable preprojective (respectively, preinjective) objects.
We say that $M\in\Ob\cat A$ is preprojective (respectively, preinjective) if all its indecomposable summands are 
preprojective (respectively, preinjective), or, equivalently, 
if $(\tau^+)^n M=0$ (respectively, $(\tau^-)^n M=0$) for some~$n\ge 0$.
For an indecomposable preprojective $P$ (respectively, preinjective $I$) we denote by
$\delta^+(P)$ (respectively, $\delta^-(I)$) the non-negative integer, easily seen to be unique, such that $(\tau^+)^{\delta_+(P)} P$ is 
projective (respectively, $(\tau^-)^{\delta_-(I)} I$ is injective). If an indecomposable~$X$ is not preprojective (respectively, preinjective)
we set $\delta^+(X)=\infty$ (respectively, $\delta^-(X)=\infty$). Clearly we can regard $\delta^\pm$ as functions 
$\delta^\pm:\Ind\cat A\to \mathbb Z_{\ge 0}\cup\{\infty\}$.
We will need the following properties of preprojectives and preinjectives (the details can be found in \cite{ARS}*{\S VIII.1--2})
\begin{proposition}\label{prop:preinj-preproj}
\begin{enumerate}[{\rm(a)}]
 \item\label{prop:preinj-preproj.a} For any $n\ge 0$, the sets $\{\alpha\in\Ind\cat A\,:\,\delta^+(\alpha)\le n\}$ and $\{\alpha\in\Ind\cat A\,:\,\delta^-(\alpha)\le n\}$
are finite. In particular, $\widetilde{\cat P}(\cat A)$ and $\widetilde{\cat I}(\cat A)$ are countable.

Let $X$, $P$, $I$ be indecomposable, with 
$[P]\in\widetilde{\cat P}(\cat A)$ and~$[I]\in\widetilde{\cat I}(\cat A)$. Then
\item\label{prop:preinj-preproj.b} $\Hom_{\cat A}(X,P)\not=0\implies\delta^+(X)\le \delta^+(P)$. 
\item\label{prop:preinj-preproj.c} $\Ext_{\cat A}^1(P,X)\not=0\implies\delta^+(X)\le \delta^+(P)-1$.
\item\label{prop:preinj-preproj.d} $\Hom_{\cat A}(I,X)\not=0\implies \delta^-(X)\le \delta^-(I)$. 
\item\label{prop:preinj-preproj.e} $\Ext_{\cat A}^1(X,I)\not=0\implies\delta^-(X)\le \delta^-(I)-1$.
\item\label{prop:preinj-preproj.f} $P$ and $I$ are exceptional bricks.
\item\label{prop:preinj-preproj.g} If $|X|=|P|$ (respectively, $|X|=|I|$)
then $X\cong P$ (respectively, $X\cong I$).
\end{enumerate}
\end{proposition}

\begin{lemma}\label{lem:assym}
Suppose that $E$, $E'$ are non-isomorphic preprojective or preinjective indecomposables in~$\cat A$. Then $\Hom_{\cat A}(E,E')\not=0$ implies that~$\Hom_{\cat A}(E',E)=0$ and $\Ext^1_{\cat A}(E',E)\not=0$ implies $\Ext^1_{\cat A}(E,E')=0$.
\end{lemma}
\begin{proof}
We prove the statement for preprojectives only, since the duality functor implies the dual statement for preinjectives.
If $\Hom_{\cat A}(E,E')\not=0$ then $\delta^+(E)\le \delta^+(E')$ by Proposition~\ref{prop:preinj-preproj} and
so the strict inequality immediately yields $\Hom_{\cat A}(E',E)=0$. 
Assume that~$\delta^+(E)=
\delta^+(E')=r$. Thus, $E=(\tau^-)^r P$ and $E'=(\tau^-)^r P'$ where $P,P'$ are projective and indecomposable.
Since~$E\not\cong E'$, $P\not\cong P'$ by \cite{ARS}*{Proposition~VIII.1.3}.
Since the functor $\tau^-$ is fully faithful on the subcategory of modules without injective summands
(\cite{ARS}*{Lemma~VIII.1.1}), 
it is thus enough to prove that $\Hom_{\cat A}(P,P')\not=0$ implies that $\Hom_{\cat A}(P',P)=0$.
By \cite{ARS}*{Lemma II.1.12}, $\Hom_{\cat A}(P,P')\not=0$ implies that $P$ is isomorphic to 
a submodule of~$P'$ and similarly $\Hom_{\cat A}(P',P)\not=0$ implies that~$P'$ is isomorphic to
a submodule of~$P$. Since all involved modules are finite dimensional over~$\kk$, we conclude that
$P\cong P'$.
\end{proof}

\subsection{}
Let $E$ be an indecomposable preprojective object in~$\cat A$.
Let $\Ind_E^<\cat A$ (respectively, $\Ind_{E}^0\cat A$, $\Ind_E^>\cat A$) be
the subset of~$\Ind\cat A$ consisting of $\alpha$ with $\delta^+(\alpha)<\delta^+(E)$
(respectively, $\delta^+(\alpha)=\delta^+(E)$, $\delta^+(\alpha)>\delta^+(E)$). 
Set $\cat A^{>}_E=\add_{\cat A}\Ind_E^>\cat A$, $\cat A^{0}_E=\add_{\cat A}\Ind^{0}_E\cat A$
and $\cat A^{<}_E=\add_{\cat A}\Ind_E^<\cat A$ and let~$\cat A^{\ge 0}=\add_{\cat A}(\Ind^0_E\cat A\cup
\Ind^{>}_E\cat A)$. For an indecomposable preinjective object, we define similar subcategories with~$\delta^+$ 
replaced by~$\delta^-$. 

It turns out that $(\cat A^{<}_E,\cat A^0_E,\cat A^{>}_E)$ is a factorizing triple. Moreover, we have the following
\begin{proposition}
\begin{enumerate}[{\rm(a)}]
\item\label{prop:tr-ppj.a'} $\Ind_E^<\cat A$ and $\Ind_{E}^0\cat A$ are finite sets;
 \item\label{prop:tr-ppj.a} $\cat A^{>}_E$, $\cat A^{0}_E$, $\cat A^{<}_E$ and $\cat A^{\ge}_E$ are closed under extensions;
\item\label{prop:tr-ppj.b} 
The multiplication map $H_{\cat A^{<}_E}\tensor H_{\cat A^{0}_E}\tensor H_{\cat A^{>}_E}\to H_{\cat A}$ 
is an isomorphism of vector spaces
and $\Exp_{\cat A}=\Exp_{\cat A^{<}_E}\Exp_{\cat A^{0}_E}\Exp_{\cat A^{>}_E}$;
\item\label{prop:tr-ppj.c} 
For all $[M],[N]\in \Ind_{E}^0\cat A$, $[M][N]=q^r [N][M]$, $q=|\kk|$, $r\in\mathbb Z$, and 
$H_{\cat A^{0}_E}$ is generated by the $[M]\in\Ind_{E}^0\cat A$; 
\item\label{prop:tr-ppj.d}
$\mathcal S_E$ defined in Proposition~\ref{prop:auto} is an Ore set in~$H_{\cat A}$ and
$\Exp_{\cat A}([E])$ acts on $H_{\cat A}[\mathcal S_E^{-1}]$ by conjugation.
\end{enumerate}
\end{proposition}
\begin{proof}
Part~\eqref{prop:tr-ppj.a'} is immediate from Proposition~\ref{prop:preinj-preproj}\eqref{prop:preinj-preproj.a}.

Let $[M]\in\Ind_E^<\cat A$, $[N]\in\Ind_{0}\cat A$, $[K]\in\Ind_E^>\cat A$. Then $$
\Hom_{\cat A}(K,N)=\Hom_{\cat A}(K,M)=\Hom_{\cat A}(N,M)=0$$
by Proposition~\ref{prop:preinj-preproj}\eqref{prop:preinj-preproj.b},  and part~\eqref{prop:tr-ppj.a} follows from Lemma~\ref{lem:orth-cond}.
Furthermore, $\cat A=\cat A_E^<\bigvee\cat A_E^0\bigvee\cat A_E^>$.
By Proposition~\ref{prop:preinj-preproj}\eqref{prop:preinj-preproj.c}
$$
\cat A_E^< \subset (\cat A_E^{\ge 0})^{\lort},\qquad \cat A_E^0\subset (\cat A_E^{>})^{\lort}.
$$
Then~\eqref{prop:tr-ppj.b} follows from Proposition~\ref{prop:decomp}.
To prove~\eqref{prop:tr-ppj.c}, note that by  Proposition~\ref{prop:preinj-preproj}\eqref{prop:preinj-preproj.c}, $\Ext^1_{\cat A}(L,L')=0=\Ext^1_{\cat A}(L',L)$ for all $L,L'$ with $[L],[L']\in\Ind_E^{0}\cat A$ and by Lemma~\ref{lem:assym}
$\Hom_{\cat A}(L,L')\not=0$ implies that $\Hom_{\cat A}(L',L)=0$. Thus, by Corollary~\ref{cor:deg-Rie}
$$
{[L]} {[L']}=\frac{|\Hom_{\cat A}(L',L)|}{|\Hom_{\cat A}(L,L')|}{[L']}{[L]}.
$$
Since $\Hom_{\cat A}(L,L')$ is a finite dimensional $\kk$-vector space, this fraction is a power of~$q$.
In particular, every ${[N]}\in\Ind A^{0}_E$, can be written, up to a power of~$q$, as a product of the $[L]\in\Ind_{E}^0\cat A$.
To prove~\eqref{prop:tr-ppj.d}, note that by part~\eqref{prop:tr-ppj.b}, every basis element~$[M]$ of~$H_{\cat A}$ can be written
as a product of a basis element of $H_{\add_{\cat A}(E)^\rort}$ and of a basis element of $H_{\add_{\cat A}(E)^\lort}$.
It remains to apply Proposition~\ref{prop:auto}.
\end{proof}

\begin{lemma}\label{lem:norm-ord}
There exists a total order~$\prec$, called {\em normal}, on $\widetilde{\cat P}(\cat A)$ such that 
$[M]\prec[N]$ implies that $\Ext^1_{\cat A}(M,N)=0=\Hom_{\cat A}(N,M)$. Similarly, there exists a total order $\prec$ on 
$\widetilde{\cat I}(\cat A)$ such that $[M]\prec[N]$ implies that $\Ext^1_{\cat A}(N,M)=0=
\Hom_{\cat A}(M,N)$. 
\end{lemma}
\begin{proof}
Fix $n\ge 0$. We first proceed to define an order~$\prec_n$ on the set $\{\alpha\in\widetilde{\cat P}(\cat A)\,:\,
\delta^+(\alpha)=n\}$
then extend it to a total order on $\widetilde{\cat P}(\cat A)$ by
setting $\alpha\prec \beta$ if $\delta^+(\alpha)<\delta^+(\beta)$ or $\delta^+(\alpha)=\delta^+(\beta)=n$ and
$\alpha\prec_n\beta$. To define $\prec_n$, it is enough to define $\prec_0$, since $\tau^-$ is 
a fully faithful functor on the full subcategory category of~$\cat A$ whose objects have no injective summands. 
Thus, if $P$, $P'$ are projective indecomposables, set $P\prec_0 P'$ if $P$ is a subobject of~$P'$. Since 
$\cat A$ is acyclic, this can be completed to a total order on the set of projective indecomposables with the desired 
property. 
\end{proof}

\subsection{}
Suppose that~$\cat A$ is indecomposable in the sense that it cannot be written as a direct sum of abelian blocks.
Then $\Ind\cat A$ is a finite set if and only if there exists an indecomposable 
module which is both preinjective and preprojective (\cite{ARS}*{Proposition VIII.1.14}). Otherwise, there exists 
an indecomposable module, called regular, which is neither preprojective, nor preinjective.
Let $\cat R(\Lambda)$ be the set of isomorphism 
classes of such modules. More generally, a module is called regular if all its indecomposable summands are regular.

Let $\cat A_-=\add_{\cat A}\widetilde{\cat P}(\cat A)$, $\cat A_+=\add_{\cat A}\widetilde{\cat I}(\cat A)$ and $\cat A_0=
\add_{\cat A}\cat R(\Lambda)$. Thus, $\cat A_-$ (respectively, $\cat A_+$, $\cat A_0$) is the full subcategory of
preprojective (respectively, preinjective, regular) modules.  We also set $\cat A_{\ge 0}=\add_{\cat A}(\widetilde{\cat I}(\cat A)\bigsqcup
\cat R(\cat A))$ and $\cat A_{\le 0}=\add_{\cat A}(\widetilde{\cat P}(\cat A)\bigsqcup \cat R(\cat A)$; thus, objects of~$\cat A_{\ge 0}$ (respectively,
$\cat A_{\le 0}$)
are modules which have no preprojective (respectively, preinjective) summands.
Since $$\Ind\cat A=\widetilde{\cat P}(\cat A)\bigsqcup \cat R(\Lambda)\bigsqcup \widetilde{\cat I}(\cat A)$$ and 
$\Hom_{\cat A}(R,P)=\Hom_{\cat A}(I,P)=\Hom_{\cat A}(I,R)=0$ for all indecomposable $P$, $R$, $I$ such that $[P]\in \widetilde{\cat P}(\cat A)$,
$[R]\in\cat R(\Lambda)$ and $[I]\in\widetilde{\cat I}(\cat A)$, it follows from Lemma~\ref{lem:orth-cond} that 
$\cat A_\pm$, $\cat A_{\ge 0}$, $\cat A_{\le 0}$ and $\cat A_0$ are closed under extensions. 
\begin{proposition}\label{prop:factor}
Suppose that $\Ind\cat A$ is infinite. Then $H_{\cat A}\cong H_{\cat A_-}\tensor H_{\cat A_0}\tensor 
H_{\cat A_+}$ as an algebra, and $\Exp_{\cat A}=\Exp_{\cat A_-}\Exp_{\cat A_{\ge 0}}=\Exp_{\cat A_-} \Exp_{\cat A_0} 
\Exp_{\cat A_+}=\Exp_{\cat A_{\le 0}}\Exp_{\cat A_+} $. Moreover,
$$
\Exp_{\cat A_\pm}=\overset{\rightarrow}{\displaystyle\prod_{[M]\in\Ind\cat A_\pm}}\Exp_{\cat A}([M]),
$$
where the product is taken in the normal order, that is, if $[M]\prec [M']$ then $\Exp_{\cat A}([M])$
occurs to the left of~$\Exp_{\cat A}([M'])$.
\end{proposition}
\begin{proof}
Clearly, $\cat A=\cat A_-\bigvee\cat A_0\bigvee \cat A_+$.
Since by Proposition~\ref{prop:preinj-preproj}
$$
\Ext^1_{\cat A}(M_-,M_0)=\Ext^1_{\cat A}(M_-,M_+)=\Ext^1_{\cat A}(M_0,M_+)=0
$$
and 
$$
\Hom_{\cat A}(M_0,M_-)=\Hom_{\cat A}(M_+,M_-)=\Hom_{\cat A}(M_+,M_0)=0
$$
the first assertion follows from Proposition~\ref{prop:decomp}.

To prove the second, note that 
$[P]\prec [Q]$ implies that $\Hom_{\cat A}(Q^{\oplus b},P^{\oplus a})=0=\Ext^1_{\cat A}(P^{\oplus a},Q^{\oplus b})$.
Since $\widetilde{\cat P}(\cat A)$ is countable, number its elements according to the normal order
as 
$$
\widetilde{\cat P}(\cat A)=\{[P_1],[P_2],\dots\},\qquad [P_1]\prec[P_2]\prec\cdots
$$ Then  we have by
Corollary~\ref{cor:deg-Rie}
$$
{[P_1^{\oplus a_1}\oplus P_2^{\oplus a_2}\oplus\cdots]}={[P_1^{\oplus a_1}]}{[P_2^{\oplus a_1}]}\cdots
$$ 
and we obtain the desired factorization for $\Exp_{\cat A_-}$.
\end{proof}
Thus, we have
\begin{equation}\label{eq:2-fact}
\Exp_{\cat A_-}\Exp_{\cat A_0}\Exp_{\cat A_+}=\Exp_{\cat A}=\Exp_{\cat A}({[S_1]})\cdots \Exp_{\cat A}({[S_r]}),
\end{equation}
where the product in the right hand side is written in the source order.
\subsection{}\label{sec:Coxeter}
Fix an ordering of the isomorphism classes of simples, say the one corresponding to the source order.
Since $\cat A$ has the finite length property, $\mathbf S=\{|S_i|\}_{1\le i\le r}$ is a basis of~$K_0(\cat A)$
and
$$
\gr M=\sum_{i} [M:S_i] \gr{S_i}.
$$
It is well-known (cf.~\cite{ARS}*{Lemma~VIII.2.1}) that $\mathbf P=\{\gr{P_i}\}_{1\le i\le r}$ (respectively, $\mathbf I=\{\gr{I_i}\}_{1\le i\le r}$) where $P_i$ (respectively, $I_i$) is the projective cover 
(respectively, the injective envelope) of a simple object~$S_i$, are bases of~$K_0(\cat A)$.
The automorphism~$c$ of~$K_0(\cat A)$ defined by $c(\gr{P_i})=-\gr{I_i}$ is called the {\em Coxeter transformation}.

Consider the non-symmetric Euler form $(\cdot\,,\,\cdot):K_0(\cat A)\times K_0(\cat A)\to\mathbb Z$ defined by
$$
(|M|,|N|)=\dim_\kk \Hom_{\cat A}(M,N)-\dim_\kk\Ext^1_{\cat A}(M,N).
$$
Thus, $\lr{|M|,|N|}=|\kk|^{(|M|,|N|)}$.
Let~$d_i=(|S_i|\,,\,|S_i|)$. 
Since for any~$M\in\Ob\cat A$
$$
\dim_{\End_{\cat A}S_i}\Hom_{\cat A}(P_i,M)=[M:S_i]=\dim_{\End_{\cat A}S_i}\Hom_{\cat A}(M,I_i)
$$
we have
$$
(|P_i|,|S_j|)=d_i \delta_{i,j}=\sum_{k} [P_i:S_k](|S_k|,|S_j|)
$$
which implies that the basis change matrix from~$\mathbf P$ to~$\mathbf S$
is equal to $(C^T)^{-1}D$, where 
$$
C=((|S_i|,|S_j|))_{1\le i,j\le r}
$$ 
is the matrix of the Euler form 
with respect to~$\mathbf S$ and $D=\operatorname{diag}(d_1,\dots,d_r)$. 
Similarly, the basis change matrix 
from~$\mathbf I$ to~$\mathbf S$ equals~$C^{-1}D$.
Thus, we obtain
\begin{lemma}\label{lem:Cox-matrix}
The matrix of the Coxeter transformation with respect to~$\mathbf S$ equals~$-C^{-1}C^T$.
\end{lemma}
Since~$(|S_i|,|S_j|)=0$ if~$i>j$ and 
$d_i=\dim_{\kk}\End_{\cat A}S_i\not=0$, we conclude that the form~$(\cdot,\cdot)$ is non-degenerate.

We will need the following 
properties of the Coxeter transformation. 
\begin{proposition}[\cite{ARS}*{Proposition~VIII.2.2 and Corollary~VIII.2.3}]\label{prop:Coxeter}
Let $M$ be an indecomposable object.
\begin{enumerate}[{\rm(a)}]
 \item \label{prop:Coxeter.a} If~$M$ is not projective then $c(\gr M)=\gr{\tau^+ M}\in K_0^+(\cat A)$. Otherwise,
$c(\gr M)\in -K_0^+(\cat A)$.
\item \label{prop:Coxeter.b}
If~$M$ is not injective then $c^{-1}(\gr M)=\gr{\tau^- M}\in K_0^+(\cat A)$. Otherwise,
$c^{-1}(\gr M)\in -K_0^+(\cat A)$.
\item \label{prop:Coxeter.c}
For any preprojective (respectively, preinjective) object~$X$, there exists~$r> 0$ such that $c^r(\gr X)$
(respectively, $c^{-r}(\gr X)$) is in~$-K_0^+(\cat A)$.
\end{enumerate}
\end{proposition}

\begin{proof}[Proof of Lemma~\ref{lem:gamma} and Theorem~\ref{thm:descr-prepr-Hall}]
Since $\lr{\gr{P_{i}},\alpha_j}=q_i^{\delta_{ij}}=\lr{\alpha_j,\gr{I_i}}$ and the form $\lr{\cdot,\cdot}$
is non-degenerate, it follows that~$\gamma_{-i}=|P_i|$ 
and $\gamma_i=\gr{I_i}$, $1\le i\le r$. Therefore, the automorphism~$c$ defined before Lemma~\ref{lem:gamma}
coincides with the Coxeter transformation.
Then~$\gamma_{i,k}=c^k(\gr{I_i})=\gr{(\tau^+)^k I_i}$
and~$\gamma_{-i,-k}=c^{-k}(\gr{P_i})=\gr{(\tau^-)^k P_i}$ for $1\le i\le r$ and~$k\ge 0$. By Proposition~\ref{prop:Coxeter},
for each~$\gamma\in\Gamma_\pm$ there exists a unique, up to an isomorphism, indecomposable~$E_\gamma$ with $\gr{E_\gamma}=\gamma$.
Thus, $\widetilde{\cat P}(\cat A)=\{ [E_\gamma]\,:\,\gamma\in \Gamma_-\}$ and $\widetilde{\cat I}(\cat A)=\{ [E_\gamma]\,:\,\gamma\in \Gamma_+\}$.
The remaining assertions of the Lemma follow the fact that the category~$\cat A$ has finitely many 
isomorphism classes of indecomposables if and only if $\widetilde{\cat P}(\cat A)\cap \widetilde{\cat I}(\cat A)\not=\emptyset$ and this happens 
if and only if $\widetilde{\cat P}(\cat A)=\widetilde{\cat I}(\cat A)$ (cf.~\cite{ARS}*{Propositions~VIII.1.13--14}).

Moreover, the order on~$\Gamma_\pm$
defined in~Theorem~\ref{thm:descr-prepr-Hall} coincides with the normal order on
$\widetilde{\cat P}(\cat A)$ and~$\widetilde{\cat I}(\cat A)$. It only remains to apply Proposition~\ref{prop:factor}.
\end{proof}

We have 
$$
\dim_{\kk}\Hom_{\cat A}(P_i,(\tau^-)^{k}P_j)=d_i[(\tau^-)^k(P_j):S_i]=(|P_i|,c^{-k}(|P_j|))
$$
and
$$
\dim_{\kk}\Ext^1_{\cat A}((\tau^-)^k P_j,P_i)=\dim_{\kk} \Hom_{\cat A}(P_i,(\tau^-)^{k-1}P_j)=
d_i[(\tau^-)^{k-1}P_j:S_i]=(|P_i|,c^{-k+1}(|P_j|)).
$$
Therefore, we have
\begin{equation}\label{eq:dim-Hom}
\dim_{\kk}\Hom_{\cat A}(E_{\gamma_{-i,-k}},E_{\gamma_{-j,-r}})=
d_i[(\tau^-)^{r-k}(P_j):S_i]=(\gamma_{-i},c^{k-r}(\gamma_{-j})),\qquad r\ge k
\end{equation}
and, by~\eqref{eq:Ext-Hom}
\begin{equation}\label{eq:dim-Ext}
\dim_{\kk}\Ext^1_{\cat A}(E_{\gamma_{-j,-r}},E_{\gamma_{-i,-k}})=d_i[(\tau^-)^{r-k-1}(P_j):S_i]
(\gamma_{-i},c^{k+1-r}(\gamma_{-j})),\qquad r>k.
\end{equation}
This, together with~\eqref{eq:fund-rel}, \eqref{eq:fund-rel'},  can be used to write explicit commutator relations between
indecomposable preprojective objects. Namely, 
$$
(\ad^*_{q^{(\gamma_{-i},c^{k-r}(\gamma_{-j}))}(1,q^{d_i},q^{2d_i},\dots,q^{(\gamma_{-i},c^{k+1-r}(\gamma_{-j}))})}[E_{\gamma_{-i,-k}}])([E_{\gamma_{-j,-r}}])=0,
$$
and
$$
(\ad_{q^{(\gamma_{-i},c^{k-r}(\gamma_{-j}))}(1,q^{d_j},q^{2d_j},\dots,q^{(\gamma_{-i},c^{k+1-r}(\gamma_{-j}))})}{[E_{\gamma_{-j,-r}}]})({[E_{\gamma_{-i,-k}}]})=0.
$$

\subsection{}\label{sec:compalg}
We finish this section with a proof of Corollary~\ref{cor:compalg}.
\begin{lemma}\label{lem:monoid-split}
Suppose that $\cat A$ is not of finite type. Then
$K_0^+(\cat A_-)\cap K_0^+(\cat A_{\ge 0})=0=K_0^+(\cat A_{\le 0})\cap K_0^+(\cat A_+)$.
\end{lemma}
\begin{proof}
Take $\gamma\in K_0^+(\cat A_-)\cap K_0^+(\cat A_{\ge 0})$.
Since $\gamma\in K_0^+(\cat A_-)$, by Proposition~\ref{prop:Coxeter}\eqref{prop:Coxeter.c} there exists~$r>0$ such that $c^r(\gamma)\in -K_0^+(\cat A)$. 
On the other hand, again by Proposition~\ref{prop:Coxeter}, since~$\gamma\in K_0^+(\cat A_{\ge 0})$, we have $c^s(\gamma)\in K_0^+(\cat A)$ for all~$s\ge 0$ hence 
$c^r(\gamma)\in K_0^+(\cat A)\cap (-K_0^+(\cat A))=0$. Since~$c$ is an automorphism, $\gamma=0$. The second statement 
is proved similarly.
\end{proof}

Now we have all ingredients to prove Corollary~\ref{cor:compalg}. Recall that we denote by~$C_{\cat A}$ the composition algebra of~$\cat A$
and by $E_{\cat A}$ the subalgebra of~$H_{\cat A}$ generated by the homogeneous components of~$\Exp_{\cat A}$.
Note that in this case~$C_{\cat A}=E_{\cat A}$ by~\eqref{eq:2-fact}.
\begin{proposition}\label{prop:compalg}
\begin{enumerate}[{\rm(a)}]
 \item\label{prop:compalg.a} $H_{\cat A_\pm}\subset C_{\cat A}$;
\item\label{prop:compalg.b} $\Exp_{\cat A_0}|_\gamma\in C_{\cat A}$ for all $\gamma\in K_0(\cat A)_0$.
\end{enumerate}
\end{proposition}
\begin{proof}
If~$\cat A$ has finitely many indecomposables, $\cat A=\cat A_\pm$ and so $E_{\cat A_\pm}=E_{\cat A}=C_{\cat A}$. Otherwise,
$(\cat A_-,\cat A_{\ge 0})$ 
and $(\cat A_{\le 0},\cat A_+)$ 
are factorizing pairs for~$\cat A$ by Proposition~\ref{prop:factor}. Then by Lemmata~\ref{lem:monoid-split} and~\ref{lem:comp-exp}, $E_{\cat A_\pm}\subset E_{\cat A}=C_{\cat A}$.

Proposition~\ref{prop:preinj-preproj}(\ref{prop:preinj-preproj.f},\ref{prop:preinj-preproj.g}) and Corollary~\ref{cor:deg-Rie}\eqref{cor:deg-Rie.b}
imply that $\cat A_+$ and~$\cat A_-$ satisfy the assumptions of Lemma~\ref{lem:comp-exp-1}, hence~$H_{\cat A_\pm}=E_{\cat A_\pm}$ and thus are subalgebras of~$C_{\cat A}$.
This proves~\eqref{prop:compalg.a}.

To prove~\eqref{prop:compalg.b}, it remains to observe that by~\eqref{eq:2-fact}
\begin{equation*}
\Exp_{\cat A_0}=\Exp_{\cat A_-}{}^{-1} \Exp_{\cat A}\Exp_{\cat A_+}{}^{-1}.\qedhere
\end{equation*}
\end{proof}

\section{Examples}\label{sec:examples}

\subsection{}\label{sec:k-spec}

We will now discuss the case of a hereditary acyclic category with only two non-isomorphic simples in more detail. Thus,
$\mathcal E_A$ is the valued graph $1\xrightarrow{(a_1,a_0)} 0$, $a_0a_1>0$.
Then $S_1=I_1$, $S_0=P_0$, $\gr{P_1}=\alpha_{10}=\alpha_1+a_0\alpha_0$ and $\gr{I_0}=\alpha_{01}=a_1\alpha_1+\alpha_0$. We 
have $$\widetilde{\cat P}(\cat A)=\{[P_{2n+i}]\}_{0\le n\le \delta^-(P_i),i\in I},\qquad 
\widetilde{\cat I}(\cat A)=\{[I_{2n+i}]\}_{0\le n\le\delta^+(I_i),i\in I}
$$ 
where 
$P_{2n+i}=(\tau^-)^n(P_i)$ and $I_{2n+i}=(\tau^+)^n I_i$. In this case, it is possible to describe 
the images of preinjectives and preprojectives in~$K_0(\cat A)$ very explicitly via a rather simple recursion.
Consider the Auslander-Reiten quiver of the preprojective component (\cite{ARS}*{Propositions~VIII.1.15--16})
$$
\xymatrix@C=2ex@R=2em{&[P_1]\ar[rd]^{(a_0,a_1)}&&[P_3]\ar@{-->}[ll]\ar[rd]^{(a_0,a_1)}&&\cdots\ar@{-->}[ll]\\[P_0]\ar[ru]^{(a_1,a_0)}&&[P_2]\ar@{-->}[ll]\ar[ru]_{(a_1,a_0)}&&
[P_4]\ar@{-->}[ll]\ar[ru]&&\cdots\ar@{-->}[ll]}
$$
where the dashed arrows denote the Auslander-Reiten translation.
Thus, for all~$k\ge 0$, we have short exact sequences 
$$
0\to P_{2k}\to P_{2k+1}^{\oplus a_1}\to P_{2k+2}\to 0,\qquad 0\to P_{2k+1}\to P_{2k+2}^{\oplus a_0}\to P_{2k+3}\to 0,
$$
whence $|P_{r+1}|+|P_{r-1}|=a_r |P_r|$, $r\ge 1$, where $a_r=a_{r\pmod 2}$. Similar considerations yield a recursion $|I_{r+1}|+|I_{r-1}|=a_r |I_r|$, $r\ge 1$.
Combining these, we obtain
\begin{lemma}
Let~$\beta_0=\alpha_0$, $\beta_{-1}=-\alpha_1$ and define $\beta_n$, $n\in\mathbb Z\setminus\{-1,0\}$ by
$$
\beta_{n+1}+\beta_{n-1}=a_n\beta_n.
$$
Then
$$
\beta_{2r}=\begin{cases}
             \gr{P_{2r}},& 0\le r< \delta^-(P_0)\\
-\gr{I_{-2r+2}},&-\delta^+(I_0)\le r<0
            \end{cases},
\qquad 
\beta_{2r+1}=\begin{cases}
               \gr{P_{2r+1}},&0\le r<\delta^-(P_1)\\
-\gr{I_{-2r-1}},&-\delta^+(I_1)\le r<0.
              \end{cases}
$$
\end{lemma}
\noindent
Thus, we can identify $\widetilde{\cat P}(\cat A)$ with $\{\beta_r\,:\, r\ge 0\}\subset K_0(\cat A)$ and
$\widetilde{\cat I}(\cat A)$ with $\{-\beta_r\,:\, r<0\}\subset K_0(\cat A)$.

It is not hard to write an explicit formula for~$\beta_r$, $r\in\mathbb Z$. Let~$U_n$ be the Chebyshev polynomial of the second kind
$$
U_n(t)=2t U_{n-1}-U_{n-2},\qquad n\ge 1,\qquad U_{-1}=0,\quad U_0=1
$$
and set $\lambda_n(t)=U_{n-1}(t/2-1)$, $\mu_n(t)=U_n(t/2-1)+U_{n-1}(t/2-1)$. Then
$$
\beta_r=a_{r+1} \lambda_{\lfloor(r+1)/2\rfloor}(a_0a_1)\alpha_{r+1}+\mu_{\lfloor r/2\rfloor}(a_0a_1)\alpha_{r},\qquad r\in\mathbb Z,
$$
where~$\alpha_k=\alpha_{k\pmod 2}$.

If $r=s\pmod 2$, $r\le s$ we have (in this case $\End P_r\cong\End P_s$)
\begin{align*}
&\dim_{\End P_r}\Hom_{\cat A}(P_r,P_s)=\mu_{(s-r)/2}(a_0a_1)\\
&\dim_{\End P_r}\Ext^1_{\cat A}(P_s,P_r)=\mu_{(s-r)/2-1}(a_0a_1).
\end{align*}
while for $r<s$ with $r=s+1\pmod 2$ 
\begin{alignat*}{2}
&\dim_{\End P_r}\Hom_{\cat A}(P_r,P_s)=a_r \lambda_{(s+1-r)/2}(a_0a_1),&\quad&\dim_{\End P_s}\Hom_{\cat A}(P_r,P_s)=a_s \lambda_{(s+1-r)/2}(a_0a_1)\\
&\dim_{\End P_r}\Ext^1_{\cat A}(P_s,P_r)=a_r \lambda_{(s-r-1)/2}(a_0a_1),&\quad&\dim_{\End P_s}\Hom_{\cat A}(P_r,P_s)=a_s \lambda_{(s-r-1)/2}(a_0a_1).
\end{alignat*}
where we used~\eqref{eq:dim-Hom}, \eqref{eq:dim-Ext}.
This, together with~\eqref{eq:fund-rel} and~\eqref{eq:fund-rel'}, allows us to write the commutation relations among
all preprojective objects. Namely, if $r\le s$ and $r=s\pmod 2$
\begin{equation}\label{eq:fund-rel-preproj}
\sum_{j=0}^{x+1} (-1)^j q_r^{\binom{j}2+jy}\qbinom[q_r]{x+1}{j} {[P_r]}^j {[P_s]}{[P_r]}^{x+1-j}=0=
\sum_{j=0}^{x+1} (-1)^j q_r^{\binom{j}2-jy}\qbinom[q_r]{x+1}{j} {[P_s]}^{x+1-j} {[P_r]}{[P_s]}^j
\end{equation}
where $x=\mu_{\frac12(s-r)-1}(a_0a_1)$ and~$y=\mu_{(s-r)/2}(a_0a_1)$. If~$r<s$ and $s=r+1\pmod 2$,
\begin{multline}\label{eq:fund-rel-preproj'}
\sum_{j=0}^{a_r z+1} (-1)^j q_r^{\binom{j}2+ja_r w}\qbinom[q_r]{a_r z+1}{j} {[P_r]}^j {[P_s]}{[P_r]}^{a_r z+1-j}=0\\=
\sum_{j=0}^{a_s z+1} (-1)^j q_s^{\binom{j}2-ja_s w}\qbinom[q_s]{a_s z+1}{j} {[P_s]}^{a_s z+1-j} {[P_r]}{[P_s]}^j,
\end{multline}
where $z=\lambda_{(s-r-1)/2}(a_0a_1)$ and~$w=\lambda_{(s-r+1)/2}(a_0a_1)$.

Both sets $\widetilde{\cat P}(\cat A)$, $\widetilde{\cat I}(\cat A)$ are finite (and hence~$\Ind\cat A$ 
is finite)
if and only if~$a_0a_1<4$. The normal order on~$\widetilde{\cat P}(\cat A)$ 
(respectively, $\widetilde{\cat I}(\cat A)$) is given by $\beta_r\prec \beta_s$ (respectively, $\beta_{-s}\prec\beta_{-r}$)
if~$r<s$. If~$a_0a_1\ge 4$, the identity
\eqref{eq:2-fact} can be written as follows
$$
\Big(\overset{\rightarrow}{\displaystyle\prod_{r>0}}\Exp_{q_r}(E_{\beta_r})\Big)
\Exp_{\cat A_0} \Big(\overset{\leftarrow}{\displaystyle\prod_{r>1}}\Exp_{q_r}(E_{-\beta_{-r}})\Big)=
[\Exp_{q_0}(E_{\alpha_0}),\Exp_{q_1}(E_{\alpha_1})]
$$
where $[a,b]=a^{-1}bab^{-1}$.

If~$a_0a_1<4$, we have $\cat A_+=\cat A_-=\cat A$ and 
$$
\Exp_{\cat A}=\overset{\rightarrow}{\displaystyle\prod_{r\ge 0}}\Exp_{q_r}(E_{\beta_r})=
\overset{\leftarrow}{\displaystyle\prod_{r> 0}}\Exp_{q_r}(E_{-\beta_{-r}})
$$
(both products are finite)
which yields
$$
\overset{\rightarrow}{\displaystyle\prod_{r>0}}\Exp_{q_r}(E_{\beta_r})=[\Exp_{q_0}(E_{\alpha_0}),\Exp_{q_1}(E_{\alpha_1})].
$$
More explicitly, we have 
$$
[\exp_{q_0}(E_{\alpha_0}),\exp_{q_1}(E_{\alpha_1})]=\begin{cases}\exp_{q_1}(E_{\alpha_{10}}),&a_0=a_1=1\\
\exp_{q_1}(E_{\alpha_{10}})\exp_{q_0}(E_{\alpha_{01}}),
&a_0a_1=2\\
\exp_{q_1}(E_{\alpha_{10}})
\exp_{q_0}(E_{\alpha_{01}+\alpha_0})\exp_{q_1}(E_{\alpha_1+\alpha_{10}})\exp_{q_0}(E_{\alpha_{01}}),& a_0a_1=3.
\end{cases}
$$
It should be noted that these Chevalley-type relations hold in higher ranks.

\subsection{}\label{sec:ex-pent}
Consider the category $\cat A$ of~$\kk$-representations of the quiver
$$
1\to 2\to\cdots\to n.
$$
If~$\alpha_i=\gr{S_i}$, then isomorphism classes of indecomposable objects are uniquely determined by 
their images in $K_0(\cat A)$ which are given by $\alpha_{i,j}=\sum_{k=i}^j\alpha_k$. Denote the corresponding 
indecomposable by $E_{i,j}$; in particular, $E_{i,i}=S_i$. Then 
$$
\dim_\kk\Hom_{\cat A}(E_{j,k},S_i)=\delta_{i,j},\qquad \dim_{\kk}\Hom_{\cat A}(S_i,E_{j,k})=\delta_{i,k}
$$
and
$$
\dim_{\kk}\Ext^1_{\cat A}(S_i,E_{j,k})=\delta_{i,j-1},\qquad \dim_{\kk}\Ext^1_{\cat A}(E_{j,k},S_i)=\delta_{i,k+1}.
$$
Fix~$1\le i<n$ and let~$\cat A_+$ (respectively, $\cat A_-$) be the full subcategory of~$\cat A$ whose objects 
satisfy $[M:S_j]=0$ if~$j\le i$ (respectively, $j>i$). It is immediate from Corollary~\ref{cor:sorder} that 
$(\cat A_+,\cat A_-)$ is a factorizing pair. We claim that it in fact a pentagonal pair. Indeed, clearly 
$\cat A_-\subset \cat A_+{}^{\lort}$. Moreover, if~$M$ is an indecomposable object that is not in 
$\cat A_-\bigvee \cat A_+$, then $M\cong E_{j,k}$ with either~$j\le i$ or~$k>i$. In particular,
every such object is in~$\cat A_0=\cat A_-^{\rort}\cap \cat A_+^{\lort}$ and so $\cat A=\cat A_-\bigvee
\cat A_0\bigvee \cat A_+$. It is now easy to see, using Proposition~\ref{prop:decomp}, that $(\cat A_-,\cat A_0,\cat A_+)$ is a factorizing triple
and hence $(\cat A_+,\cat A_-)$ is a pentagonal pair.

Also, let~$E=E_{\alpha_{ij}}$. Then $\cat A^0_E=\add_{\cat A}\{[\alpha_{k,j}]\,:\,k<j\}$,
$\cat A^<_E=\add_{\cat A}\{[\alpha_{rs}]\,:\, s>j\}$ and finally $\cat A^>_E=\add_{\cat A}\{[\alpha_{rs}]\,:\,s<j\}$.

\subsection{}
We now discuss a non-hereditary example. Consider the following quiver
$$
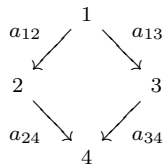
\let\objectstyle\scriptstyle
\xymatrix@=3ex{&1\ar[ld]_{a_{12}}\ar[rd]^{a_{13}}\\2\ar[rd]_{a_{24}}&&3\ar[ld]^{a_{34}}\\&4}
$$
with relation $a_{34}a_{13}=a_{24}a_{12}$. Let~$\cat A$ be the category of finite dimensional 
representations of that quiver over~$\kk$ with $|\kk|=q$. Denote~$\alpha_i=|S_i|$, $1\le i\le 4$. 
The isomorphism classes of 
indecomposables in~$\cat A$ are uniquely determined by their images in~$K_0(\cat A)$.
The Auslander-Reiten quiver of~$\cat A$  is (cf.~\cite{ARS}*{\S VII.2})
$$
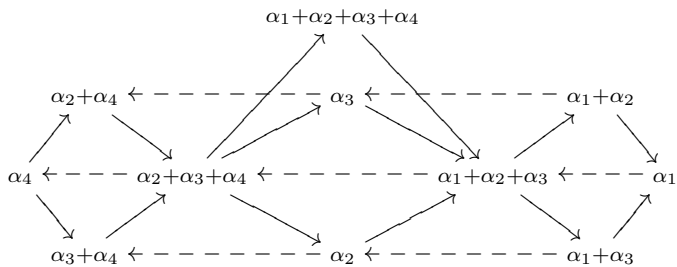
\let\objectstyle\scriptstyle
\xymatrix@C=-.1em@R=3.5ex{&&&\alpha_1+\alpha_2+\alpha_3+\alpha_4\ar[rdd]\\&\alpha_2+\alpha_4\ar[rd]&&\alpha_3\ar[rd]\ar@{-->}[ll]&&\alpha_1+\alpha_2\ar[rd]\ar@{-->}[ll]\\
\alpha_4\ar[rd]\ar[ru]&&\alpha_2+\alpha_3+\alpha_4\ar@{-->}[ll]\ar[rd]\ar[ru]\ar[ruu]&&\alpha_1+\alpha_2+\alpha_3\ar[ru]\ar[rd]\ar@{-->}[ll]&&\alpha_1\ar@{-->}[ll]\\
&\alpha_3+\alpha_4\ar[ru]&&\alpha_2\ar@{-->}[ll]\ar[ru]&&\alpha_1+\alpha_3\ar@{-->}[ll]\ar[ru]}
$$
where dashed arrows denote the Auslander-Reiten translation. Then we have
\begin{multline*}
\exp_q(E_4)\exp_q(E_{24})\exp_q(E_{34})\exp_q(E_{234})\exp_q(E_2)\exp_q(E_3)\exp_q(E_{1234})\exp_q(E_{123})\times\\
\exp_q(E_{12})\exp_q(E_{13})\exp_q(E_1)=
\exp_q(E_1)\exp_q(E_2)\exp_q(E_3)\exp_q(E_4)
\end{multline*}
where $E_{i_1i_2...}$ denotes $[\alpha_{i_1}+\alpha_{i_2}+\cdots]$. In particular, if 
we set~$\cat A_+=\add_{\cat A}S_1$ and let $\cat A_-$ be the largest full subcategory of~$\cat A$
whose objects satisfy $[M:S_1]=0$ then $(\cat A_+,\cat A_-)$ is a pentagonal pair by Proposition~\ref{prop:decomp} (note
that $\cat A_-$ is equivalent to the category of $\kk$-representations of the quiver~$2\to 4\leftarrow 3$).

Consider now the same quiver
$$
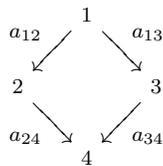
\let\objectstyle\scriptstyle
\xymatrix@=3ex{&1\ar[ld]_{a_{12}}\ar[rd]^{a_{13}}\\2\ar[rd]_{a_{24}}&&3\ar[ld]^{a_{34}}\\&4}
$$
but this time with the relation $a_{24}a_{12}=0$. The Auslander-Reiten quiver in this case is 
$$
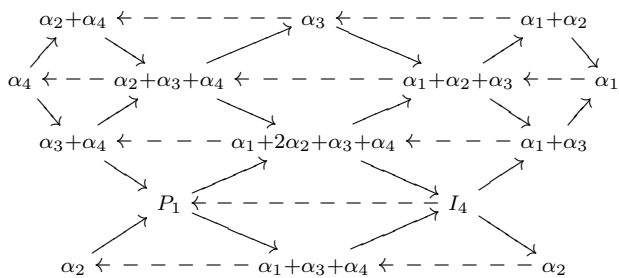
\let\objectstyle\scriptstyle
\xymatrix@C=-.5em@R=2ex{&\alpha_2+\alpha_4\ar[rd]&&\alpha_3\ar@{-->}[ll]\ar[dr]&&\alpha_1+\alpha_2\ar@{-->}[ll]\ar[rd]\\\alpha_4\ar[ur]\ar[rd]&&\alpha_2+\alpha_3+\alpha_4\ar[rd]\ar@{-->}[ll]\ar[ur]&&\alpha_1+\alpha_2+\alpha_3\ar[rd]\ar@{-->}[ll]\ar[ur]&&\alpha_1\ar@{-->}[ll]\\
&\alpha_3+\alpha_4\ar[ru]\ar[rd]&&\alpha_1+2\alpha_2+\alpha_3+\alpha_4\ar[ru]\ar[rd]\ar@{-->}[ll]&&\alpha_1+\alpha_3\ar[ru]\ar@{-->}[ll]\\&&P_1\ar[ru]\ar[rd]&&I_4
\ar[ru]\ar[rd]\ar@{-->}[ll]\\&\alpha_2\ar[ru]&&\alpha_1+\alpha_3+\alpha_4\ar@{-->}[ll]\ar[ru]&&\alpha_2\ar@{-->}[ll]}
$$
where the two copies of~$S_2$ in the lowest row are to be identified and $P_1$ (respectively, $I_4$) is the projective cover 
(respectively, the injective envelope) of~$S_1$ (respectively, $S_4$), are non-isomorphic and 
$|P_1|=|I_4|=\alpha_1+\alpha_2+\alpha_3+\alpha_4$. In this case, no partition of the set of isomorphism classes of simples 
gives a pentagonal pair, since the indecomposable $X$ with~$|X|=\alpha_1+\alpha_3+\alpha_4$ satisfies
$\Ext^1_{\cat A}(X,S_2)\not=0$ and~$\Ext^1_{\cat A}(S_2,X)\not=0$, while it involves all other simples as its composition factors.

\subsection{}
Suppose that we have have an autoequivalence on~$\cat A$ which induces 
a permutation $\sigma$ of the set of isomorphism classes of simples in~$\cat A$ and a
natural action of~$\sigma$ on~$K_0(\cat A)$. Thus, on the level of~$\mathcal E_{\cat A}$,
$\sigma$ must induce an automorphism of valued graphs. 
Then we can consider 
the full subcategory $\cat A^\sigma$ of~$\cat A$ whose objects~$M$ satisfy 
$\gr{M}\in K_0(\cat A)^\sigma$. This category is clearly closed under extensions. 

The simple 
objects in~$\cat A^\sigma$ are of the form $S_{\mathbf i}=\bigoplus_{i\in \mathbf i} S_i$ where $\mathbf i\in \{1,\dots,r\}/\sigma$.
There is a natural source order on the set of $\sigma$-orbits and we get 
$$
\Exp_{\cat A^\sigma}=\Exp_{\cat A^\sigma}(S_{\mathbf i_1})\cdots\Exp_{\cat A^\sigma}(S_{\mathbf i_r}).
$$

\subsection{}\label{sec:Jord}
Let~$R$ be a principal ideal domain such that~$R/\mathfrak m$ is a finite field for every maximal ideal~$\mathfrak m$.
Consider the category~$\cat A$ of $R$-modules of finite length. It is well-known that 
$\cat A=\bigoplus_{\mathfrak m\in\operatorname{Spec}R} \cat A(\mathfrak m)$ where 
$\cat A(\mathfrak m)$ is the full subcategory of~$\cat A$ whose objects are
finite length $R$-modules $M$ with~$\operatorname{Ann}M=\mathfrak m^j$ for some~$j\ge 0$.
Each of the categories $\cat A(\mathfrak m)$ is hereditary, and $\Hom_{\cat A}(M,N)=\Ext^1_{\cat A}(M,N)=0$ if
$M\in\Ob\cat A(\mathfrak m)$, $N\in\Ob\cat A(\mathfrak m')$ with $\mathfrak m\not=\mathfrak m'$.
Thus, 
$$
\Exp_{\cat A}=\prod_{\mathfrak m\in\Spec R} \Exp_{\cat A(\mathfrak m)}.
$$
Fix $\mathfrak m\in\Spec R$. Then for each~$r>0$, there exists a unique indecomposable~$\mathcal I_r:=\mathcal I_r(\mathfrak m)=R/\mathfrak m^r$ of length~$r$, and
for every partition $\lambda=(\lambda_1\ge \lambda_2\ge \cdots)$ there is a unique object $\mathcal I_\lambda=\mathcal I_\lambda(\mathfrak m)=\mathcal I_{\lambda_1}\oplus
\mathcal I_{\lambda_2}\oplus\cdots$.

The Hall algebra of $\cat A(\mathfrak m)$ is in fact very well understood (cf., for example, \cites{Schiffmann,Mac}). It is commutative, is isomorphic
to the Hall algebra of the category of nilpotent finite length modules over $\kk[x]$ where $\kk\cong R/\mathfrak m$
and is the classical Hall-Steiniz algebra. It is freely generated by ${[\mathcal I_{(1^r)}]}={[\mathcal I_1^{\oplus r}]}$.
Then
$$
{\Exp_{\cat A(\mathfrak m)}}^{-1}=\sum_{r\ge 0} (-1)^r q^{\binom{r}{2}} {[\mathcal I_{(1^r)}]},
$$
where $q=|R/\mathfrak m|$.
There is a well-known homomorphism $\Phi:H_{\cat A(\mathfrak m)}\to \operatorname{Sym}$, where $\operatorname{Sym}$ is the ring 
of symmetric polynomials in infinitely many variables
$x_1,x_2,\dots$,
given by $\Phi([\mathcal I_{(1^r)}])\mapsto q^{-\binom{r}2}e_r$ where $e_r=\sum_{1\le i_1<i_2<\cdots<i_r} x_{i_1}\cdots x_{i_r}$ is 
the $r$th elementary symmetric function. Let $\widehat{\operatorname{Sym}}$ be the natural completion of~$\operatorname{Sym}$
and extend $\Phi$ to an isomorphism $\widehat H_{\cat A(\mathfrak m)}\to \widehat{\operatorname{Sym}}$. Then
$$\Phi({\Exp_{\cat A(\mathfrak m)}}^{-1})=
\sum_{r\ge 0} (-1)^r e_r=\prod_{i\ge 1} (1-x_i).
$$
This implies that 
$$
\Phi({\Exp_{\cat A(\mathfrak m)}})=\prod_{i\ge 1}\frac{1}{1-x_i}=\sum_{r\ge 0} h_r,
$$
where $h_r$ is the $r$th complete symmetric function.
Let 
$$
P_\lambda(x_1,\dots,x_r;q)=\prod_{i=0}^r [m_i]_q!^{-1} \sum_{\sigma\in S_r} \sigma\Big( x_1^{\lambda_1}\cdots x_r^{\lambda_r} \prod_{1\le i<j\le r}
\frac{x_i-q x_j}{x_i-x_j}\Big)
$$
where $\lambda=(\lambda_1\ge\cdots\ge \lambda_r\ge 0)$ and $m_i=\#\{ 1\le j\le r\,:\,\lambda_j=i\}$.
Since $\Phi([\mathcal I_{\lambda}])=q^{-n(\lambda)} P_\lambda(x;q^{-1})$,
where $n(\lambda)=\sum_i (i-1)\lambda_i$,
we obtain
$$
\sum_\lambda q^{n(\lambda)} P_\lambda(x;q) \sum_{r\ge 0}(-1)^r e_r=1.
$$
This identity specializes to a well-known identity in the ring of symmetric polynomials in finitely many variables (in which the second sum 
becomes finite).
In particular (cf.~\cite{Mac}),
$$
\sum_{\lambda\vdash r} q^{n(\lambda)} P_\lambda(x;q)=h_r. 
$$

It is not hard to see that
$$
|\Aut(\mathcal I_\lambda)|=q^{\sum_{i,j} (\min(i,j)-\delta_{ij})a_i a_j} \prod_{i} |\GL(a_i,\kk)|
$$
where $\lambda=(1^{a_1}2^{a_2}\cdots)$. Let $a_\lambda(q)=|\Aut(\mathcal I_\lambda)|$. Then
the homomorphism $\int:\widehat H_{\cat A(\mathfrak m)}\to \mathbb Q[[t]]$ is given by 
$$
[\mathcal I_\lambda]\mapsto a_\lambda(q)^{-1} t^{|\lambda|}
$$
and yields 
$$
\Big(\sum_{r\ge 0} t^r \sum_{\lambda\vdash r} a_\lambda(q)^{-1} \Big) \Big(\sum_{s\ge 0} (-1)^s q^{\binom{s}2}\frac{t^s}{|\GL(s,q)|}\Big)=1
$$
or
$$
\sum_{s=0}^r 
\Big(\prod_{j=1}^s (1-q^j)^{-1}\Big)
\sum_{\lambda\vdash r-s} a_\lambda(q)^{-1}=0,\qquad r>0.
$$

\subsection{}
Retain the notation of~\ref{sec:k-spec} and assume that $\cat A$ is of tame representation type.
In that case, $\cat A_0$ is abelian and is a direct sum of countably many abelian subcategories $\cat A_\rho$, known as stable tubes,
since for any $\rho$ there exists $r_\rho\ge 0$ such that 
for any indecomposable $M\in\Ob\cat A_\rho$, $\tau^{r_\rho} M\cong M$. The minimal~$r_\rho$
with that property is called the rank of~$\cat A_\rho$, and it is well-known that~$r_\rho=1$ for all
but finitely many~$\rho$.
Thus,
$$
H_{\cat A_0}\cong \bigotimes_\rho H_{\cat A_\rho}
$$
as an algebra
and moreover 
$$
\Exp_{\cat A_0}=\prod_\rho \Exp_{\cat A_\rho}.
$$
There exists an indecomposable $S_\rho\in\Ob\cat A_\rho$ of minimal length which is a
simple object in~$\cat A_\rho$. In particular, it is a brick. Then all non-isomorphic simple objects in~$\cat A_\rho$ are 
of the form $(\tau^-)^r S_\rho$, $0\le r<r_\rho$. Since $\tau$ is an exact autoequivalence $\cat A_\rho\to\cat A_\rho$,
it induces an automorphism of
$H_{\cat A_\rho}$ of order~$r_\rho$ defined by ${[M]}\mapsto {[\tau M]}$. Thus, we have
$$
\Exp_{\cat A_\rho}{}^{-1}=\sum_{\mathbf a=(a_1,\dots,a_{r_\rho})\in \mathbb Z_{\ge 0}^{r_\rho}}(-1)^{\sum_i a_i} q_\rho^{\sum_i \binom{a_i}{2}}
{[S_\rho(\mathbf a)]}
$$
where $q_\rho=|\End_{\cat A}S_\rho|$ and $S_\rho(\mathbf a)=\bigoplus_{i=1}^{r_\rho} (\tau^i S_\rho)^{\oplus a_i}$.
In particular, if~$r_\rho=1$, $\cat A_\rho$ is equivalent to the category of nilpotent representations of~$\kk[x]$ as 
described in~\ref{sec:Jord}.
\appendix
\section{Coproduct}
\label{sect:Appendix} 

Let $\cat A$ be a finitary category and define a linear map $\Delta:H_{\cat A} \to H_{\cat A}  \widehat \otimes H_{\cat A}$  by:
\begin{equation}
\label{eq:coproductDeltaprim}
\Delta([C])= \sum_{[A],[B]\in {\Iso\cat A}} \frac{|{\Hom_{\cat A}}(A,B)|} {|\Ext^1_{\cat A}(A,B)|}\,  F^{A,B}_C\cdot  [A]\otimes [B] \ ,
\end{equation}
where $F^{A,B}_C$ is the dual Hall number given by 
$$
F^{A,B}_C=\frac{|\Aut_{\cat A}(A)| |\Aut_{\cat A}(B)|}{|\Aut_{\cat A}(C)|} \, F_{A,B}^C
=\frac{|\Ext^1_{\cat A}(A,B)_C|}{|\Hom_{\cat A}(A,B)|}
$$  
where we used~\eqref{eq:Rie}.

\begin{lemma}\label{lem:coprod} We have 
$\Delta(\Exp_{\cat A})=\Exp_{\cat A}\otimes \Exp_{\cat A}$.
\end{lemma} 

\begin{proof} Since for all $A,B\in\Ob\cat A$
$$\sum_{[C] \in {\Iso\cat A}} \,\frac{|{\Hom_{\cat A}}(A,B)|} {|\Ext^1_{\cat A}(A,B)|}
F^{[A],[B]}_{[C]} =\sum_{[C]\in\Iso\cat A}\frac {|\Ext^1_{\cat A}(A,B)_C|}{|{\Ext^1_{\cat A}}(A,B)|}=1$$
we obtain
\begin{align*}
\Delta(\Exp_{\cat A})&=\sum_{[C]\in {\Iso\cat A}} \Delta([C])=\sum_{[A],[B],[C]\in {\Iso\cat A}} \frac{|{\Hom_{\cat A}}(A,B)|} {|\Ext^1_{\cat A}(A,B)|} F^{[A],[B]}_{[C]} [A]\otimes [B]\\
&=\sum_{[A],[B]\in {\Iso\cat A}} \left (\sum_{[C]\in {\Iso\cat A}}\frac{|{\Hom_{\cat A}}(A,B)|} {|\Ext^1_{\cat A}(A,B)|} F^{[A],[B]}_{[C]}\right) [A]\otimes [B]\\&=\sum_{[A],[B]\in {\Iso\cat A}} [A]\otimes [B]=\Exp_{\cat A}\otimes \Exp_{\cat A}.\qedhere
\end{align*}
\end{proof}

In general $\Delta$ is not coassociative. However, it is easy to see that $\Delta$ becomes coassociative if~$\cat A$ is hereditary. Moreover, in
that case it is a homomorphism of braided algebras. 
In order to make this statement precise, we need to introduce some terminology from the theory of braided categories.

Let ${\cat C}$ be an $\FF$-linear braided tensor category with the braiding operator $\Psi_{U,V}:U\otimes V\to V\otimes U$ for all objects $U,V$ of ${\cat C}$.  
If~$\mathcal A$, $\mathcal B$ are associative algebras in~$\cat C$, then $\mathcal A\tensor\mathcal B$ acquires a natural structure 
of an associative algebra in~$\cat C$ via $m_{\mathcal A\tensor\mathcal B}:=(m_{\mathcal A}\tensor m_{\mathcal B})(1_{\mathcal A}\tensor
\Psi_{\mathcal B,\mathcal A}\tensor 1_{\mathcal B})$, where $m_{\mathcal A}$, $m_{\mathcal B}$ are the respective multiplication morphisms.

Our main example of a category ${\cat C}$ is as follows. Let $\Gamma$ be an additive monoid and let $\chi:\Gamma\times\Gamma\to \kk^\times$ be a bicharacter. Let ${\cat C}_\chi$ be the tensor category of $\Gamma$-graded vector spaces $V=\bigoplus\limits_{\gamma\in \Gamma} V(\gamma)$ such that each component  $V(\gamma)$ is finite-dimensional. The following Lemma is immediate.

\begin{lemma} 
The category ${\cat C}_\chi$ is a braided tensor category with the braiding  $\Psi_{U,V}:U\otimes V\to V\otimes U$ for each objects $U$ and $V$ given by
$$\Psi_{U,V}(u_\gamma\otimes v_{\delta})=\chi(\gamma,\delta) \cdot v_{\delta}\otimes u_\gamma$$
for any $u_\gamma\in U(\gamma)$, $v_\delta\in V(\delta)$.   
\end{lemma}
In particular, if $\mathcal A$, $\mathcal B$ are associative algebras in~$\cat C_{\chi}$, the multiplication in~$\mathcal A\tensor\mathcal B$ is given by
$$
(a\tensor b)(a'\tensor b')=\chi(\beta,\alpha')(aa'\tensor bb'),
$$
where $a'\in \mathcal A(\alpha')$, $b\in \mathcal B(\beta)$.

Now let ${\cat A}$ be a hereditary finitary category. Let  $\Gamma=K_0({\cat A})$ be the Grothendieck group of ${\cat A}$ and 
let $\chi=\chi_{\cat A}: \Gamma\times \Gamma\to \QQ$ be the bicharacter given by:
$$\chi_{\cat A}(|M|,|N|)=\frac{|\Hom_{\cat A}(N,M)|} {|{\Ext^1_{\cat A}}(N,M)|},\qquad M,N\in\Ob\cat A.
$$
(the bicharacter $\chi_{\cat A}$ is well-defined because it is essentially  the Euler form on $K_0({\cat A})$). Thus, the coproduct~$\Delta$
can be written as
$$
\Delta([C])=\sum_{[A],[B]} \chi(|B|,|A|)F^{A,B}_C [A]\tensor [B].
$$

The following is a reformulation of Green's theorem (\cite{Green}) for Hall algebras.
\begin{theorem}  
\label{th:Hopf structures}
Let ${\cat A}$ be a finitary hereditary  category. 
Then $H_{\cat A}$ is a bialgebra in ${\cat C}_{\chi_{\cat A}}$ with the coproduct $\Delta$ given by \eqref{eq:coproductDeltaprim}.
\end{theorem}

\begin{proof}
We have
\begin{align*}
\Delta(&[C]) \Delta([C'])\\
&=\Big(\sum_{[A],[B]}  \chi(\gr{B},\gr{A})F^{A,B}_C\cdot  [A]\otimes [B]\Big)
\Big(\sum_{[A'],[B']}  \chi(\gr{B'},\gr{A'})
 F^{A',B'}_{C'}\cdot  [A']\otimes [B']\Big)\\
&=\sum_{[A],[B],[A'],[B']}  \chi(\gr{B},\gr{A}) \chi(\gr{B'},\gr{A'}) \chi(\gr{B},\gr{A'})
F^{A,B}_C F^{A',B'}_{C'} \cdot   [A] [A']\otimes [B][B']\\
&=\sum_{\substack{[A],[B],[A'] \\ [B'],[A''],[B'']}} \chi(\gr{B},\gr{A}) \chi(\gr{B'},\gr{A'}) \chi(\gr{B},\gr{A'}) F^{A,B}_C F^{A',B'}_{C'} F_{A,A'}^{A''}F_{B,B'}^{B''}\cdot  [A'']\otimes [B'']
\end{align*}

On the other hand, 
$$\Delta([C][C'])=\sum_{[C'']}F_{C,C'}^{C''} \Delta([C''])=\sum_{[C''], [A''], [B'']} \chi(\gr{B''},\gr{A''})
 F_{C,C'}^{C''} F^{A'',B''}_{C''}\cdot  [A'']\otimes [B'']$$

We need to compare the coefficients of $[A'']\tensor [B'']$ in these expressions.
Since $\gr{A''}=\gr{A}+\gr{A'}$ and~$\gr{B''}=\gr{B}+\gr{B'}$,
$$
\chi(\gr{B},\gr{A}) \chi(\gr{B'},\gr{A'}) \chi(\gr{B},\gr{A'})\chi(\gr{B''},\gr{A''})^{-1}=\chi(\gr{B'},\gr{A})^{-1}
$$
By \cite{Green}*{Theorem~2}
\begin{equation}\label{E:Greenproof3}
\sum_{[A],[B],[A'],[B']\in\Iso\cat A}\chi(|B'|,|A|)^{-1}
\cdot F^{A''}_{A,A'}F^{B''}_{B,B'} F_{C'}^{A',B'}F_C^{A,B}=\sum_{[C'']\in\Iso\cat A}    F_{C,C'}^{C''} F_{C''}^{A'',B''}
\end{equation}
for any objects $A'',B'',C,C'$ in~$\cat A$, and it is now immediate
that
$$\Delta([C])\Delta([C'])=\Delta([C][C'])
$$
that is, $\Delta$ is a homomorphism of algebras.
\end{proof}

\end{document}